\RequirePackage{color}
\documentclass[bj,numbers]{imsart}

\usepackage{amsthm,amsmath,amsfonts,amssymb,pdfpages}

\RequirePackage[numbers]{natbib}
\RequirePackage[colorlinks,citecolor=blue,urlcolor=blue]{hyperref}
\RequirePackage{graphicx}

\usepackage{graphicx} 
\graphicspath{ {./Pictures/} }
\usepackage{subfig}
\usepackage{mathtools}
\usepackage{enumitem}
\usepackage{bm, bbm}

\startlocaldefs
\numberwithin{equation}{section}
\theoremstyle{plain}
\newtheorem{theorem}{Theorem}[section] 
\newtheorem{lemma}[theorem]{Lemma}
\newtheorem{corollary}[theorem]{Corollary}
\newtheorem{proposition}[theorem]{Proposition}

\theoremstyle{remark}
\newtheorem{definition}[theorem]{Definition}

\newtheorem{example}{Example}
\newtheorem{remark}[theorem]{Remark}

\newcommand{\dI}{\Delta^{m,k}_1}
\newcommand{\dII}{\Delta^{m,k}_2}
\newcommand{\dIp}{\Delta^{m',k'}_1}
\newcommand{\dIIp}{\Delta^{m',k'}_2}

\endlocaldefs

\begin{document}
\begin{frontmatter}
\title{H\"older regularity and roughness: \\ construction and examples}
\runtitle{H\"older regularity and roughness}

\begin{aug}
\author[A]{\fnms{Erhan}~\snm{Bayraktar}\ead[label=e1]{erhan@umich.edu}\orcid{0000-0002-1926-4570}}
\author[B]{\fnms{Purba}~\snm{Das}\ead[label=e2]{purba.das@kcl.ac.uk}\orcid{0000-0002-2692-969X}}
\author[A]{\fnms{Donghan}~\snm{Kim}\ead[label=e3]{donghank@umich.edu}\orcid{0000-0003-1451-5032}}
\address[A]{Department of Mathematics, University of Michigan \printead[presep={,\ }]{e1,e3}}

\address[B]{Department of Mathematics, King's College London\printead[presep={,\ }]{e2}}
\end{aug}

\begin{abstract}
We study how to construct a stochastic process on a finite interval with given `roughness' and its finite joint moments. We first extend Ciesielski's isomorphism (1960) along a general sequence of partitions, and provide a characterization of H\"older regularity of a function in terms of the coefficients along Schauder basis. Using this characterization we propose a better (pathwise) estimator of H\"older exponent. As an additional application, we construct fake (fractional) Brownian motions with some path properties and finite moments same as (fractional) Brownian motions. These belong to non-Gaussian families of stochastic processes which are statistically difficult to distinguish from real (fractional) Brownian motions.
\end{abstract}


\end{frontmatter}


\section{Introduction} \label{sec. intro}
The degree of roughness of a function (or a sample path of a process) is often an important property when analysing financial time-series data; the H\"older exponent plays a crucial role in determining the level of roughness. For example, in the rough volatility literature, the fundamental question is how to model or estimate the roughness of the underlying volatility process, and this roughness is understood as the H\"older exponent (more precisely, critical H\"older exponent, see Definition \ref{def: holder exponent}).

Even though (the reciprocal of) the H\"older exponent of a path is theoretically different from the variation index along a given partition sequence (that is, the infimum value $p$ such that the $p$-th variation of the path is finite, see Definition \ref{def. variation index}, or \cite{das2022theory, Schied2022} for more details), these two concepts are linked through the self-similarity index in the case of fractional Brownian motions. This is why we often see that one estimates the variation index of a given path, instead of estimating the H\"older exponent (for example in rough volatility literature). However, without prior information on the process, there is no reason to believe why the H\"older exponent should coincide with the reciprocal of the variation index. In this regard, we provide explicit constructions of functions and processes which possess H\"older exponent different from the reciprocal of variation index (Example \ref{ex.different holder and variation index} and Remark \ref{rem.make random}).

In order to establish these examples, we use the Haar basis \cite{haar1910} and Faber-Schauder system \cite{schauder,semadeni1982}; some of earlier works \cite{schied2016b, schied2016} study how to construct a function with given pathwise quadratic variation ($2$-th variation), using the Haar and Faber-Schauder representation associated with the dyadic partition sequence. Since most of the financial time-series data are observed over non-uniform time intervals, the extension to a general non-uniform partition other than the dyadic partition would be useful in practice. Following the methods of Cont and Das \cite{das2021}, we extend the Schauder-type representation of a continuous function to a large class of refining partition sequences along which we define an orthonormal `non-uniform' Haar basis and a corresponding Schauder system. More concretely, any continuous function $x$ defined on a compact interval $[0, T]$ has a unique (generalized) Schauder representation
\begin{equation*}
	x(t) = x(0) + \big(x(T)-x(0)\big)t+ \sum_{m=0}^{\infty} \sum_{k \in I_m} \theta^{x, \pi}_{m,k} e^{\pi}_{m,k}(t), \qquad t \in [0, T],
\end{equation*}
along a finitely refining (Definition \ref{def.finite.refining}) partition sequence $\pi = (\pi_n)$ of $[0, T]$, where each $e^{\pi}_{m, k}$ is a Schauder function depending only on $\pi$ and $\theta^{x, \pi}_{m, k}$ is corresponding Schauder coefficient with an explicit expression in terms of $x$ and partition points of $\pi$ (Proposition~\ref{prop:coeff_hat_func}).

We then characterize the H\"older exponent of $x$ in terms of Schauder coefficients $\theta^{x, \pi}_{m, k}$ (Theorem \ref{main.thm}), which is a generalization of Ciesielski's isomorphism \citep{Ciesielski:isomorphism}, along a general refining partition sequence $\pi$ with certain conditions, namely balanced (Definition \ref{def.balance}) and complete refining (Definition \ref{def.complete refining}). This means that we can construct a process with any given critical H\"older exponent (i.e., any `roughness order') by controlling the Schauder coefficients, and this result is particularly useful when we can only observe high-frequency data (or discrete) points of a continuous process along a particular time sequence (e.g. from the financial market), but cannot observe it as an entire sample path. Moreover, we show that the variation index of a process along a given partition sequence may not be equal to the reciprocal of H\"older exponent, by constructing the aforementioned examples. Our characterization of H\"older exponent also provides a new way to estimate the H\"older regularity of a given function in terms of Schauder coefficients along a general sequence of partitions (Theorem \ref{thm.holder.estimate}), which is similar to the classical H\"older estimator from \cite{Jaffard04}, but subject to less approximation error.

In the case of fractional Brownian motions~(fBMs), its critical H\"older exponent (Definition~\ref{def: holder exponent}) coincides with Hurst index $H$ (and also equal to the reciprocal of variation index along a certain sequence of partitions). This yields the Schauder representation of fBMs along general partition sequences with any given Hurst index, where the Schauder coefficients are given as Gaussian random variables. Furthermore, we show that finite (joint) moments between the Schauder coefficients of a process are closely related to finite moments of the process; when constructing a process via Schauder representation, if we choose the Schauder coefficients to have the same mean and covariance structure as the Schauder coefficients of fBMs, then the resulting process will have the same mean and covariance as that of fBMs. This leaves us to explore the following question: can we come up with processes that are not fBM but statistically difficult to distinguish from fBMs? Our paper answers this question in an affirmative way (Subsection \ref{subsec. examples BM}, \ref{subsec. examples fBM}).

Historically non-martingale, non-Markovian processes are used in different fields including biology \citep{collins1995, filatova2008}, engineering \citep{leland1994}, physics \citep{sadhu2021}, geophysics \citep{gurbatov1997}, telecommunication networks \citep{taqqu1997network}, and
economics \citep{lo1991}. Very recently non-Markovian processes have become more and more common in finance; fBM with Hurst index $H< 1/2$ has been widely used for modeling volatility. When observing such financial data, often the infinite-dimensional distribution is not available in practice, but finitely many marginals are. These marginal distributions are determined by computing (joint) moments of data samples. If one can generate processes with the same first few moments as real (fractional) Brownian motions, those `fake' processes can easily be misdiagnosed as a (fractional) Brownian motion. Therefore, our answer to the above question suggests that one has to be careful when modeling a certain process as a (fractional) Brownian motion.

It turns out that a similar question has also been explored to construct fake versions of standard Brownian motion in the context of option pricing, where both the fake process and the standard Brownian motion are martingales. Madan and Yor \cite{madan2002} first construct a discontinuous martingale based on the Az\'ema–Yor solution of the Skorokhod embedding problem which has the same marginals as Brownian motion. Later on, Hamza and Klebaner \cite{hamza2007}, Albin \cite{albin2008}, Oleszkiewicz \cite{oleszkiewicz2008}, and more recently  Hobson \cite{hobson2016, hobson2013}, Beiglb\"ock, Lowther, Pammer, and Schachermayer \cite{Beiglbock2021} studied fake Brownian motion and fake martingales. We would like to remark that despite the same name \textit{fake}, we mimic (fractional) Brownian motions in a very different sense in this paper, and our method goes beyond semimartingales.

Section \ref{sec. Schauder representation, review} reviews generalized Schauder representation of continuous functions introduced in \cite{das2021}, and provides sufficient conditions on Schauder coefficients which ensure the continuity of functions or processes (Lemma \ref{lem: continuity condition}). Section \ref{sec. Holder and Schauder} provides a generalization of Ciesielski’s isomorphism (Theorem \ref{main.thm}) and a pathwise characterization of H\"older regularity of a continuous path (Theorem \ref{thm.holder.estimate}). Section \ref{sec. critical Holder and variation index} introduces concepts of critical H\"older exponent and variation index, and gives concrete examples of function and process for which variation index is different from the reciprocal of H\"older exponent (Example \ref{ex.different holder and variation index} and Remark \ref{rem.make random}). Section \ref{sec. mimicking BM} and \ref{sec. mimicking fBM} provide a method of constructing fake Brownian motion (Theorem \ref{thm.fake-BM}) and fractional Brownian motions (Theorem \ref{thm.fake-fBM}), respectively. In Section \ref{sec. moment}, we show that the finite (joint) moments of the Schauder coefficients uniquely determine the finite (joint) moments of the distribution of the process up to the same order (Theorem \ref{thm. finite moments}). This result allows us to mimic fBM up to any finite moments. Section \ref{sec: conclusion} provides concluding remarks.

\section{Schauder representation along general partition sequence} \label{sec. Schauder representation, review}
This section recalls the Schauder representation of a continuous function on compact support from an orthonormal non-uniform Haar basis along a general partition sequence, based on the recent results from Section 3 of \cite{das2021}. The proofs of the lemmata in this section can be found in Supplement A \cite{suppA}.

\subsection{Sequence of partitions}

For a fixed $T>0$, we shall consider a (deterministic) sequence of partitions $\pi=(\pi^n)_{n \in \mathbb{N}}$ of $[0,T]$
\begin{equation*}
	\pi^n = \left(0=t^n_1<t^n_2<\cdots<t^n_{N(\pi^n)}=T\right),
\end{equation*}
where we denote $N(\pi^n)$ the number of intervals in the partition $\pi^n$ and
\begin{equation*}
	\underline{\pi^n} := \inf_{i = 0, \cdots, N(\pi^n)-1} \vert t^n_{i+1} - t^n_i \vert, \qquad \qquad
	|\pi^n| := \sup_{i = 0, \cdots, N(\pi^n)-1} \vert t^n_{i+1} - t^n_i \vert,
\end{equation*}
the size of the smallest and the largest interval of $\pi^n$, respectively. By convention, $\pi^0 = \{0, T\}$. For example, the dyadic sequence of partitions, denoted by $\mathbb{T} \equiv \pi$, contains partition points $t^n_k = k/2^n$ for $n \in \mathbb{N}$, $k = 0, \cdots, 2^nT$. We shall now recall some relevant definitions from \cite{das2020,das2021}. 

\begin{definition} [Refining sequence of partitions]
	A sequence of partitions $\pi=(\pi^n)_{n \in \mathbb{N}}$ is said to be \textit{refining (or nested)}, if $t\in \pi^m$ implies $t\in \cap_{n\geq m} \pi^n$ for every $m \in \mathbb{N}$. In particular, we have $\pi^1\subseteq \pi^2 \subseteq \cdots$. 
\end{definition}


\begin{definition} [Finitely refining sequence of partitions]    \label{def.finite.refining}
	A sequence of partitions $\pi=(\pi^n)_{n \in \mathbb{N}}$ is said to be \textit{finitely refining}, if $\pi$ is refining with vanishing mesh, i.e., $|\pi^n| \rightarrow 0$ as $n \rightarrow \infty$, and there exists $ M < \infty$ such that $N(\pi^1) \le M$ and the number of partition points of $\pi^{n+1}$ within any two consecutive partition points of $\pi^n$ is always strictly bigger than zero and bounded above by $M$, irrespective of $n\in \mathbb{N}$. In particular, we have $\sup_{n \in \mathbb{N}} \frac{N(\pi^n)}{M^n}\leq 1$.
\end{definition}

\begin{definition} [Balanced sequence of partitions] \label{def.balance}
	A sequence of partitions $\pi=(\pi^n)_{n \in \mathbb{N}}$ is said to be \textit{balanced}, if there exists a constant $c>1$ such that $|\pi^n| \le c \underline{\pi^n}$ holds for every $n \in \mathbb{N}$.
\end{definition}

Every interval in the partition $\pi^n$ of a balanced sequence $\pi$ is asymptotically comparable. Note also that any balanced sequence $\pi$ satisfies for every $n \in \mathbb{N}$
\begin{equation}    \label{eq.wellbalanced}
	|\pi^n|\leq c \, \underline{\pi^n} \leq \frac{cT}{N(\pi^n)}.
\end{equation}
We next impose a condition on the mesh sizes of two consecutive levels.

\begin{definition} [Complete refining sequence of partitions]   \label{def.complete refining}
	A refining sequence of partitions $\pi=(\pi^n)_{n \in \mathbb{N}}$ is said to be \textit{complete refining}, if there exist positive constants $a$ and $b$ such that
	\begin{equation}    \label{def: complete refining}
		1+a \leq \frac{|\pi^n|}{|\pi^{n+1}|} \leq b  \quad \text{holds for every } n \in \mathbb{N}.
	\end{equation}
\end{definition}

\begin{remark} [Notation]
	Throughout this paper, we use the same symbols $a, b, c$, and $M$ to refer to the constants that appeared in Definitions \ref{def.finite.refining} -- \ref{def.complete refining}.
\end{remark}

The following result illustrates some straightforward properties of refining sequence of partitions.
\begin{lemma}   \label{lemma.prop.refining}
	Let $\pi = (\pi^n)_{n \in \mathbb{N}}$ be a refining sequence of partitions.
	\begin{enumerate} [label=(\roman*)]
		\item If $\pi$ is complete refining, then the mesh size $|\pi^n|$ converges exponentially fast to zero.
		\item If $\pi$ is both finitely refining and balanced, then
		\begin{equation*}
			\limsup_{n \in \mathbb{N}} \frac{|\pi^n|}{\underline{\pi^{n+1}}}<\infty.
		\end{equation*}
		In this case, we have the upper bound in \eqref{def: complete refining}, i.e., $|\pi^n| \leq b|\pi^{n+1}|$ for every $n \in \mathbb{N}$.
		\item If $\pi$ is both balanced and complete refining, then it is finitely refining.
	\end{enumerate}
\end{lemma}

\subsection{Generalized Haar basis associated with a finitely refining partition sequence}

Let us fix a finitely refining sequence $\pi$ of partitions with vanishing mesh on $[0, T]$ and denote $p(n,k) := \inf \{j \geq 0 : t^{n+1}_j \geq t^n_k \}$. Since $\pi$ is refining, the following inequality holds for every $k = 0, \cdots, N(\pi^n)-1$
\begin{equation*}
	0\leq t^{n}_{k}= t^{n+1}_{p(n,k)}<t^{n+1}_{p(n,k)+1}<\cdots <t^{n+1}_{p(n,k+1)}= t^{n}_{k+1}\leq T.    
\end{equation*}
With the notation $\Delta^{n}_{i, j} := t^n_j - t^n_i$, we now define the generalized Haar basis associated with the refining sequence $\pi$.

\begin{definition} [Generalized Haar basis] \label{def. generalized Haar basis}
	The \textit{generalized Haar basis} associated with a finitely refining sequence $\pi=(\pi^n)_{n \in \mathbb{N}}$ of partitions is a collection of piecewise constant functions $\{\psi^{\pi}_{m,k,i} \, : \, m=0,1,\cdots, ~ k=0,\cdots,N(\pi^m)-1, ~ i = 1,\cdots, p(m,k+1)-p(m,k)\}$ defined as follows:
 \footnotesize
	\begin{equation*}
		\psi^{\pi}_{m,k,i}(t)= 
		\begin{cases}
			\qquad \qquad \qquad \qquad 0, &\quad\text{if } t\notin \left[t^{m+1}_{p(m,k)},t_{p(m,k)+i}^{m+1}\right)
			\\
			\quad \left( \frac{\Delta^{m+1}_{p(m, k)+i-1, p(m, k)+i}}{\Delta^{m+1}_{p(m, k), p(m, k)+i-1}} \times \frac{1}{\Delta^{m+1}_{p(m, k), p(m, k)+i}} \right)^{\frac{1}{2}}, &\quad\text{if } t\in\left[t_{p(m,k)}^{m+1},t_{p(m,k)+i-1}^{m+1}\right)
			\\
			-\left( \frac{\Delta^{m+1}_{p(m, k), p(m, k)+i-1}}{\Delta^{m+1}_{p(m, k)+i-1, p(m, k)+i}} \times \frac{1}{\Delta^{m+1}_{p(m, k), p(m, k)+i}}\right)^{\frac{1}{2}}, &\quad\text{if } t\in \left[t_{p(m,k)+i-1}^{m+1},t_{p(m,k)+i}^{m+1}\right)
		\end{cases}.
	\end{equation*}
\end{definition}

The function $\psi^{\pi}_{m, k, i}$ is constant on the intervals $[t_{p(m,k)}^{m+1},t_{p(m,k)+i-1}^{m+1})$ and $[t_{p(m,k)+i-1}^{m+1},t_{p(m,k)+i}^{m+1})$, zero elsewhere, and satisfies $\int \psi^{\pi}_{m, k, i}(t)dt = 0$ and $\int (\psi^{\pi}_{m, k, i}(t))^2 dt = 1$. Also, $t^{m+1}_{p(m,k)+i-1}\in \pi^{m+1}\backslash\pi^{m}$ for all $i$ and $t^{m+1}_{p(m,k)}=t^m_k \in \pi^m \cap \pi^{m+1}$. Since $\pi$ is finitely refining, we have $1 < p(m,k+1)-p(m,k)\leq M <\infty$ for all $m, k$.

\subsection{Schauder representation of a continuous function}
The Schauder functions $e^{\pi}_{m,k,i} : [0,T] \rightarrow \mathbb{R}$ are obtained by integrating the generalized Haar basis
\begin{equation*}
	e^{\pi}_{m,k,i}(t) := \int_0^t \psi_{m,k,i}(s)ds =\left(\int_{t^{m+1}_{p(m,k)}}^{t\wedge t^{m+1}_{p(m,k)+i}} \psi_{m,k,i}(s)ds\right) \mathbbm{1}_{[t^m_k,t^{m+1}_{p(m,k)+i}]}(t), 
\end{equation*}
where $\mathbbm{1}$ represents the indicator function.
\begin{definition}  [Schauder function]
	For every $m,k,i$, the function $e^\pi_{m,k,i}$ is continuous, but not differentiable, and expressed as
	\footnotesize
 \begin{align*}
			e^{\pi}_{m,k,i}(t) =
			\begin{cases}
				\qquad \qquad \qquad \qquad 0, &\quad\text{if } t\notin \left[t_{p(m,k)}^{m+1},t_{p(m,k)+i}^{m+1}\right) 
				\\
				\left( \frac{\Delta^{m+1}_{p(m, k)+i-1, p(m, k)+i}}{\Delta^{m+1}_{p(m, k), p(m, k)+i-1}} \times \frac{1}{\Delta^{m+1}_{p(m, k), p(m, k)+i}} \right)^{\frac{1}{2}}\times(t-t^{m+1}_{p(m,k)}), &\quad\text{if } t\in\left[t_{p(m,k)}^{m+1},t_{p(m,k)+i-1}^{m+1}\right)
				\\
				\left( \frac{\Delta^{m+1}_{p(m, k), p(m, k)+i-1}}{\Delta^{m+1}_{p(m, k)+i-1, p(m, k)+i}} \times \frac{1}{\Delta^{m+1}_{p(m, k), p(m, k)+i}}\right)^{\frac{1}{2}}\times(t^{m+1}_{p(m,k)+i}-t), &\quad\text{if } t\in \left[t_{p(m,k)+i-1}^{m+1},t_{p(m,k)+i}^{m+1}\right)
			\end{cases}.
	\end{align*}
\end{definition}
We note that the function $e^{\pi}_{m,k,i}$ has the following bound
\begin{align}
	\max_{t \in [0, T]} |e^{\pi}_{m, k, i}(t)| &= e^{\pi}_{m, k, i}(t_{p(m, k)+i-1})
	= \Bigg(\frac{\Delta^{m+1}_{p(m, k), p(m, k)+i-1}}{\Delta^{m+1}_{p(m, k), p(m, k)+i}} \Delta^{m+1}_{p(m, k)+i-1, p(m, k)+i} \bigg)^{\frac{1}{2}}    \nonumber
	\\
	&\le \big(  \Delta^{m+1}_{p(m, k)+i-1, p(m, k)+i} \big)^{\frac{1}{2}}
	\le |\pi^{m+1}|^{\frac{1}{2}},     \label{bound of maximum Schauder}
\end{align}
since the triangle-shaped $e^{\pi}_{m, k, i}$ peaks at $t_{p(m,k)+i-1}^{m+1}$, and $\Delta^{m+1}_{p(m, k), p(m, k)+i-1} \leq \Delta^{m+1}_{p(m, k), p(m, k)+i}$.

For any finitely refining sequence $\pi$ of partitions, two families of functions $\{\psi^{\pi}_{m,k,i}\}_{m,k,i}$ and $\{e^{\pi}_{m,k,i}\}_{m,k,i}$ can be reordered as $\{\psi^{\pi}_{m,k}\}_{m,k}$ and $\{e^{\pi}_{m,k}\}_{m,k}$; for each level $m\in \{0,1,\cdots\}$, the values of $k$ run from $0$ to $N(\pi^{m+1})-N(\pi^m)-1$ after reordering, thus we shall denote $I_m := \{ 0, 1, \cdots, N(\pi^{m+1})-N(\pi^{m})-1\}$ for every $m \ge 0$. We also note that for any fixed $m \geq 0$, at most $M$ many of $\{e^\pi_{m,k}(t)\}_{k \in I_m}$ have non-zero values, where the $M$ is the constant in Definition \ref{def.finite.refining}. Thus, we have the bound for any fixed $t \in [0, T]$
\begin{equation}    \label{ineq.boundofsumofe.pi}
	\sum_{k \in I_m} \big\vert e^{\pi}_{m,k}(t) \big\vert \leq M \vert \pi^{m+1} \vert^{\frac{1}{2}}.
\end{equation}

The following result shows that any continuous function with support $[0, T]$ admits a unique Schauder representation.

\begin{proposition} [Theorem 3.8 of \cite{das2021}] \label{prop:coeff_hat_func}
	Let $\pi$ be a finitely refining sequence of partitions of $[0,T]$. Then, any continuous function $x :[0,T] \rightarrow \mathbb{R}$ has a unique Schauder representation:
	\begin{equation*}
		x(t) = x(0) + \big(x(T)-x(0)\big)t+ \sum_{m=0}^{\infty} \sum_{k \in I_m} \theta^{x, \pi}_{m,k} e^{\pi}_{m,k}(t),
	\end{equation*}
	with a closed-form representation of the Schauder coefficient
	\begin{equation}  \label{eq.theta.coeff}
		\theta^{x, \pi}_{m,k} = \frac{\big(x(t^{m,k}_{2})-x(t^{m,k}_{1})\big)(t^{m,k}_{3}-t^{m,k}_{2})-\big(x(t^{m,k}_{3})- x(t^{m,k}_{2})\big)(t^{m,k}_{2}-t^{m,k}_1)}{\sqrt{(t^{m,k}_2-t^{m,k}_1)(t^{m,k}_3-t^{m,k}_2)(t^{m,k}_3-t^{m,k}_1)}},
	\end{equation}
	where the support of the Schauder function $e^\pi_{m,k}$ is denoted by $[t^{m,k}_1,t^{m,k}_3]$ and its maximum is attained at $t^{m,k}_2$ for every $m, k$.
\end{proposition}

Conversely, with the observation that a Schauder representation up to the $n$-th level, denoted by $x_n$ in \eqref{def.Schauder.n}, is a continuous function for every $n \ge 0$, the following result collects some conditions on the Schauder coefficients $\theta^{x, \pi}_{m, k}$ such that the sequence $(x_n)_{n \ge 0}$ converges to a continuous function as $n \rightarrow \infty$.

Here and in what follows, we shall use a fixed probability space $(\Omega, \mathcal{F}, \mathbb{P})$ on which the Schauder coefficients $\theta^{x, \pi}_{m, k}$ are defined, if they are given as random variables. Moreover, we shall write $\theta_{m, k}$, $\theta^{x}_{m, k}$, or $\theta^{\pi}_{m, k}$, instead of $\theta^{x, \pi}_{m, k}$, if the notation is clear from context.

\begin{lemma} [Continuous limit]   \label{lem: continuity condition}
	Consider the sequence $(x_n)_{n \ge 0}$ of Schauder representation up to level $n$
	\begin{equation}    \label{def.Schauder.n}
		x_n(t) := x_0(t)+\sum_{m=0}^{n} \sum_{k \in I_m} \theta_{m,k} e^{\pi}_{m,k}(t),
	\end{equation}
	where $x_0(t)$ is a linear function. If $\pi$ is balanced and complete refining, and the coefficients $\{\theta_{m,k}\}_{k \in I_m, m \ge 0}$ are random variables satisfying either one of the following conditions:
	\begin{enumerate} [label=(\roman*)]
		\item $\{\theta_{m,k}\}_{k \in I_m, m \ge 0}$ have a uniformly bounded fourth moment, i.e., $\mathbb{E}[\theta^4_{m, k}] \leq M_0 < \infty$ for all $k \in I_m$, $m \ge 0$,
		or
		\item there exist some positive constants $C_1, C_2$, and $\epsilon$ such that the inequality
		\begin{equation}    \label{con.theta bound delta}
			\mathbb{P} \Big[ \vert \theta_{m, k}\vert  \vert\pi^{m+1} \vert^{ \frac{1}{2}-\epsilon} > \delta  \Big]
			\le C_1 \exp \Big( -\frac{\delta^2}{2C_2} \Big)
		\end{equation}
		holds for every $\delta > 0$, $k \in I_m$, and $m \ge 0$, or
		\item there exist some constants $\epsilon >0$ and $C>0$ such that the bound $\vert \theta_{m, k} \vert \leq C \vert \pi^{m+1} \vert^{\epsilon-\frac{1}{2}}$ holds for every $k \in I_m$ and $m \ge 0$,
	\end{enumerate}
	then $x_n(t)$ converges uniformly in $t$ to a continuous function $x(t)$ almost surely.
\end{lemma}

\section{H\"older exponent and Schauder coefficients}   \label{sec. Holder and Schauder}

This section shows that the Schauder coefficients of a continuous function along a fixed sequence $\pi$ of partitions of $[0, T]$ are related to the H\"older exponent. We shall denote $C^{0}([0,T],\mathbb{R})$ the space of continuous functions defined on $[0, T]$, $C^{\alpha}([0,T],\mathbb{R})$ the space of $\alpha$-H\"older continuous functions for $0 < \alpha < 1$, i.e.,
\begin{equation*}
	C^{\alpha}([0,T],\mathbb{R}) := \bigg\{ x\in C^0([0,T], \mathbb{R}) ~ \bigg| ~ \sup_{\substack{t,s \in [0,T] \\ t\neq s }}\frac{\vert x(t)-x(s) \vert}{|t-s|^{\alpha}} < \infty  \bigg\},
\end{equation*}
and $C^{\nu-}([0,T],\mathbb{R}) := \cap_{0\leq \alpha< \nu}C^{\alpha}([0,T],\mathbb{R})$ the space of functions which are $\alpha$-H\"older continuous for every $\alpha<\nu$.
If $x\in C^{\alpha}([0,T],\mathbb{R})$, we denote the $\alpha$-H\"older (semi)-norm of $x$ as
\begin{equation}    \label{def.Holder norm}
	\Vert x \Vert_{C^{\alpha}([0, T])} := \sup_{\substack{t,s \in [0,T] \\ t\neq s }}\frac{\vert x(t)-x(s) \vert}{|t-s|^{\alpha}}.
\end{equation}

\begin{remark}
	In the above class $C^{\alpha}([0,T],\mathbb{R})$, the expression $\Vert x \Vert_{C^{\alpha}([0, T])}$ is not a norm but a semi-norm. In order to make it a norm we need to add an extra term, such as $\Vert x \Vert_{\infty} := \sup \{t\in[0,T]: |x(t)|\}$, or $|x(0)|$. Since this paper only considers continuous functions over a compact interval $[0,T]$, these terms are always bounded, and it does not affect the finiteness of either H\"older norm or H\"older semi-norm. Thus, we shall use the semi-norm instead of the H\"older norm throughout the paper. 
\end{remark}

We provide some preliminary results. The proofs of Lemmas \ref{lemma.sumability}, \ref{lem.difference.e} are given in Supplement A \cite{suppA}.

\begin{lemma}   \label{lemma.sumability}
	For any complete refining sequence $\pi$, we have for every $n_0 \in \mathbb{N}$ and $0 < \alpha < 1$
	\begin{equation*}
		\sum_{m=0}^{n_0} |\pi^{m+1}|^{\alpha-1} \le  K^{\alpha}_1 |\pi^{n_0+1}|^{\alpha-1} \qquad \text{and} \qquad \sum_{m=n_0}^\infty |\pi^{m+1}|^{\alpha} \le K^{\alpha}_2 |\pi^{n_0+1}|^{\alpha},
	\end{equation*}
	where $K^{\alpha}_1 := \frac{1}{1-(1+a)^{\alpha-1}}$ and $K^{\alpha}_2 := \frac{1}{1-(1+a)^{-\alpha}}$, with the constant $a$ from \eqref{def: complete refining}.
\end{lemma}

\begin{lemma}   \label{lem.difference.e}
	For any balanced and finitely refining partition sequence $\pi$ and any given two points $t \neq t' \in [0, T]$, let us consider a unique non-negative integer $n_0$ satisfying $\vert \pi^{n_0+1} \vert < \vert t-t'\vert \leq \vert \pi^{n_0} \vert$. Then, the sum of $m$-th level difference in Schauder functions $e^\pi$ at these two points has the following bounds with the constants from Definitions \ref{def.finite.refining} and \ref{def.balance}
	\begin{equation*}
		\sum_{k \in I_m} \big\vert e^{\pi}_{m,k}(t)- e^{\pi}_{m,k}(t') \big\vert \le
		\begin{cases} 
			2M \vert \pi^{m+1} \vert^{\frac{1}{2}},  & \text{ if } m \ge n_0,
			\\
			2M\sqrt{c} \vert \pi^{m+1} \vert^{-\frac{1}{2}}\vert t-t'\vert,   & \text{ if } m < n_0.
		\end{cases}
	\end{equation*}
\end{lemma}

\subsection{Characterization of H\"older exponent}

Based on the previous bounds, we present a connection between the H\"older semi-norm and the Schauder coefficients. Though the following result only considers a scalar-valued continuous function $x$, we note here that it can be generalized to any vector-valued continuous function. For $x = (x_1, \cdots x_d)$ where each $x_i \in C^0([0, T], \mathbb{R})$ for $i = 1, \cdots, d$, we have the bounds \eqref{main.bound.theta} of $\sup_{m,k} (\vert \theta^{x_i}_{m,k} \vert \vert \pi^{m+1} \vert^{\frac{1}{2}-\alpha})$ in terms of $\Vert x_i \Vert_{C^{\alpha}([0, T])}$ for each $x_i$, and from these bounds, we can easily derive similar bounds for the coefficient vector $\theta^x_{m, k} = (\theta^x_{m, k, 1}, \cdots, \theta^x_{m, k, d})$.

\begin{theorem}[H\"older semi-norm and Schauder coefficients] \label{main.thm}
Let $\pi$ be a balanced and complete refining  sequence of partitions $[0,T]$ and $\{\theta_{m,k}\}$ be the Schauder coefficients along $\pi$ of $x \in C^0([0,T], \mathbb{R})$. Then, $x \in C^{\alpha}([0,T], \mathbb{R})$ if and only if $\sup_{m,k}(|\theta_{m,k}||\pi^{m+1}|^{\frac{1}{2}-\alpha})$ is finite. In this case, we have the bounds
\begin{equation}    \label{main.bound.theta}
	\frac{1}{2M\sqrt{c}K^{\alpha}_1 + 2M K^{\alpha}_2} \Vert x \Vert_{C^{\alpha}([0, T])} \le  \sup_{m,k} \Big( \vert \theta_{m,k} \vert \vert \pi^{m+1} \vert^{\frac{1}{2}-\alpha} \Big) \le 2(\sqrt{c})^{3} \Vert x \Vert_{C^{\alpha}([0, T])}.
\end{equation}
\end{theorem}

\begin{proof}
From Lemma \ref{lemma.prop.refining} (iii), we first note that $\pi$ is also finitely refining. Recalling the representation \eqref{eq.theta.coeff} of $\theta_{m, k}$ in Proposition \ref{prop:coeff_hat_func} for any $m \ge 0$ and $k \in I_m$, $|\theta_{m,k}|$ has the following upper bound from the H\"older continuity of $x$ with the constant $c$ in Definition \ref{def.balance}
\begin{align*}
	\vert \theta_{m,k} \vert
	&\leq \Vert x \Vert_{C^{\alpha}([0, T])} \frac{ (t^{m,k}_{2}-t^{m,k}_{1})^{\alpha} (t^{m,k}_{3}-t^{m,k}_{2}) + (t^{m,k}_{3}- t^{m,k}_{2})^{\alpha} (t^{m,k}_{2}-t^{m,k}_1)}{\sqrt{(t^{m,k}_2-t^{m,k}_1)(t^{m,k}_3-t^{m,k}_2)(t^{m,k}_3-t^{m,k}_1)}}
	\\
	&\leq 2(\sqrt{c})^{3} \Vert x \Vert_{C^{\alpha}([0, T])} \vert \pi^{m+1} \vert^{\alpha-\frac{1}{2}}.
\end{align*}
Thus, we obtain $\vert \theta_{m,k} \vert \vert \pi^{m+1} \vert^{\frac{1}{2}-\alpha} \le 2(\sqrt{c})^{3}\Vert x \Vert_{C^{\alpha}([0, T])}$. Since the right-hand side is independent of $m$ and $k$, taking the supremum over $m, k$ in both sides to obtain the second inequality of \eqref{main.bound.theta}.

In order to show the first inequality, take two points $t\neq t' \in [0,T]$ and derive
\begin{align*}
	\frac{|x(t)-x(t')|}{|t-t'|^\alpha} &= \bigg\vert \sum_{m=0}^{\infty} \sum_{k \in I_m} \theta_{m,k} \frac{e^{\pi}_{m,k}(t) - e^{\pi}_{m,k}(t')}{|t-t'|^\alpha} \bigg\vert          \nonumber
	\\
	& \leq \sup_{m,k} \Big( \vert \theta_{m,k} \vert \vert \pi^{m+1} \vert^{\frac{1}{2}-\alpha} \Big) \sum_{m=0}^\infty\sum_{k \in I_m} \big\vert e^{\pi}_{m,k}(t) - e^{\pi}_{m,k}(t')\big\vert \frac{\vert \pi^{m+1} \vert^{\alpha-\frac{1}{2}}}{|t-t'|^\alpha}.
\end{align*}
From Lemma~\ref{lem.difference.e} and \ref{lemma.sumability}, the double summation can be bounded as
\begin{align*}
	\sum_{m=0}^\infty & \sum_{k \in I_m} \big\vert e^{\pi}_{m,k}(t) - e^{\pi}_{m,k}(t')\big\vert \frac{\vert \pi^{m+1} \vert^{\alpha-\frac{1}{2}}}{|t-t'|^\alpha}
	\\
	&\le 2M\sqrt{c} \sum_{m=0}^{n_0-1} |\pi^{m+1}|^{\alpha-1} |t-t'|^{1-\alpha} + 2M \sum_{m=n_0}^\infty |\pi^{m+1}|^{\alpha}|t-t'|^{-\alpha}
	\\
	& \le 2M\sqrt{c} K^{\alpha}_1 \vert \pi^{n_0} \vert^{\alpha-1} \vert t-t' \vert^{1-\alpha} + 2M K^{\alpha}_2 \vert \pi^{n_0+1} \vert^{\alpha} \vert t-t' \vert^{-\alpha}
	\le 2M\sqrt{c} K^{\alpha}_1 + 2M K^{\alpha}_2.
\end{align*}
Here, the last inequality follows from the relationship $\vert \pi^{n_0+1} \vert < \vert t-t'\vert \leq \vert \pi^{n_0} \vert$. Therefore, we have
\begin{equation*}
	\frac{|x(t)-x(t')|}{|t-t'|^\alpha} \le (2M\sqrt{c} K^{\alpha}_1 + 2M K^{\alpha}_2) \sup_{m,k} \Big( \vert \theta_{m,k} \vert \vert \pi^{m+1} \vert^{\frac{1}{2}-\alpha} \Big),
\end{equation*}
and the right-hand side is independent of $t$ and $t'$. Taking the supremum over all $t\neq t' \in [0, T]$ proves the first inequality.
\end{proof}

\begin{remark}
Ciesielski \cite{Ciesielski:isomorphism} stated the isomorphism between the space $C^{\alpha}([0, T], \mathbb{R})$ of $\alpha$-H\"older continuous functions and the space $\ell^{\infty}(\mathbb{R})$ of all bounded real sequences, equipped with the semi-norm $\Vert \cdot \Vert_{C^{\alpha}([0, T])}$ in \eqref{def.Holder norm} and $\Vert \xi \Vert := \sup_n \vert \xi_n \vert$ for $\xi = (\xi_1, \xi_2, \cdots )$. However, this isomorphism only considers the dyadic partition sequence on the interval $[0, T]$; when a different sequence $\pi$ of partitions is used (as long as it is balanced and complete refining), a given element of $\ell^{\infty}(\mathbb{R})$ generates a different $\alpha$-H\"older continuous function than the one constructed along the dyadic sequence. Therefore, Theorem \ref{main.thm} gives rise to the precise statement: there is a bijection between the space $C^{\alpha}([0, T], \mathbb{R})$ and the product space $\Pi_{b, c}([0, T]) \times \ell^{\infty}(\mathbb{R})$, where $\Pi_{b, c}([0, T])$ denotes the set of all balanced and complete refining sequences of partitions of $[0, T]$.
\end{remark}

\begin{corollary}   \label{cor. Holder continuity}
For a balanced and complete refining partition sequence $\pi$ on $[0,T]$, let $\{ \theta_{m,k} \}$ be the Schauder coefficients along $\pi$ of $x \in C^0([0,T], \mathbb{R})$ satisfying condition (ii) of Lemma \ref{lem: continuity condition}. Then, $x \in C^{\alpha}([0,T], \mathbb{R})$ for every $\alpha < \epsilon$.
\end{corollary}
The proof of Corollary \ref{cor. Holder continuity} can be found in Supplement A \cite{suppA}.

\subsection{Pathwise H\"older regularity estimator}

The global H\"older regularity of a given function can be represented in terms of its wavelet basis coefficients, as in the following.

\begin{proposition} [Proposition 5.1 of \cite{Echelard}, \cite{Jaffard04}]
Let $\{\psi_{m,k}\}$ be an orthonormal wavelet basis having enough zero moments (\cite[conditions (2.9)-(2.15)]{Echelard}). If the global H\"older regularity of a continuous function $f \in C^0([0,T],\mathbb{R})$ is given by $\alpha_0(f)$, then
\begin{equation}\label{eq.Holder.estimate.old}
	\alpha_0(f) + \frac{1}{2} = \liminf_m \min_k \frac{\log_2|\langle f,\psi_{m,k}\rangle|}{-m}, 
\end{equation} 
where $\langle f,\psi_{m,k}\rangle$ is the projection of $f$ along $\psi_{m,k}$, i.e., $\langle f,\psi_{m,k}\rangle = \int_\mathbb{R}f(u)\psi_{m,k}(u)du$.
\end{proposition}

Even though the identity \eqref{eq.Holder.estimate.old} is well-known and widely used for estimating the global H\"older regularity of a given function, the estimator constructed from \eqref{eq.Holder.estimate.old}, by taking up to finite terms on the right-hand side, has some disadvantages. For example, in a practical situation one does not observe the entire continuous function $f:[0,T] \rightarrow \mathbb{R}$, but observes high-frequency samples from the function $f$. Then, the projections $ \langle f,\psi_{m,k}\rangle$ of $f$ cannot be calculated directly but need to be estimated from the sample points:
\begin{equation*}
\langle f,\psi_{m,k}\rangle \approx \sum_{0 \leq t_i < T} f(t_i)\psi_{m,k}(t_i)(t_{i+1}-t_{i}).
\end{equation*}
Additionally, often time financial data are not observed over uniform time intervals. Since the above estimation of the H\"older exponent explicitly uses the Haar basis along the dyadic partitions, the observations need to be along the uniform time window. 

From Theorem \ref{main.thm}, we provide a very similar characterization of global H\"older regularity; whenever one constructs an estimator from our characterization, our estimator does not require any additional approximation for computing the projection $\langle f,\psi_{m,k}\rangle$ of $f$, and it enables us to estimate the H\"older exponent for given data observed over non-uniform time intervals.

\begin{theorem}[Characterization of H\"older regularity] \label{thm.holder.estimate}
Let us fix a balanced and complete refining (not necessarily uniform) partition sequence $\pi$ of $[0, T]$. If the global H\"older regularity of a continuous function $f$ is $\alpha_0(f)$, then
\begin{equation}    \label{holder.estimate.new}
	\alpha_0(f) - \frac{1}{2} = 
	\begin{cases}
		~~\limsup_m \max_k \frac{\log|\theta_{m,k}^{f,\pi}|}{\log|\pi^{m+1}|},  & \text{if } \limsup_m \max_{k\in I_m}\log|\theta_{m,k}^{f,\pi}| = -\infty,
		\\
		\qquad \qquad 0, & \text{if } \limsup_m \max_{k\in I_m}\log|\theta_{m,k}^{f,\pi}| = L \in (-\infty, \infty),
		\\
		- \limsup_m \max_k \frac{\log|\theta_{m,k}^{f,\pi}|}{-\log|\pi^{m+1}|}, & \text{if } \limsup_m \max_{k\in I_m}\log|\theta_{m,k}^{f,\pi}| = +\infty.
	\end{cases}        
\end{equation}
where $\theta_{m,k}^{f,\pi}$ is the unique Schauder coefficient of $f$ along $\pi$, given in \eqref{eq.theta.coeff}.
\end{theorem}

\begin{proof}
From the given regularity $\alpha_0(f)$, Theorem \ref{main.thm} shows the existence of $C_0\in (0,\infty)$ satisfying
\begin{align*}
	\sup_{m,k} \Big( \vert \theta^{f,\pi}_{m,k} \vert & \vert \pi^{m+1} \vert^{\frac{1}{2}-\alpha_0(f)} \Big) = C_0
    \iff \sup_{m,k}  \bigg( \log \vert \theta^{f,\pi}_{m,k} \vert - \Big(\alpha_0(f) - \frac{1}{2} \Big) \log \vert \pi^{m+1} \vert \bigg) = \log C_0.
\end{align*}
Furthermore, we have
\begin{align}\label{eq.alpha.estimator}
	\sup_{m,k}  \bigg( \log \vert &\theta^{f,\pi}_{m,k} \vert - \Big(\alpha_0(f) - \frac{1}{2} \Big) \log \vert \pi^{m+1} \vert \bigg)<\infty \\
	&\iff \limsup_{m} \max_{k} \bigg( \log \vert \theta^{f,\pi}_{m,k} \vert - \Big(\alpha_0(f) - \frac{1}{2} \Big) \log \vert \pi^{m+1} \vert \bigg)< \infty.    \nonumber  
\end{align}
We now consider three different cases.
\begin{enumerate} [label=(\roman*)]
	\item $\limsup_m \max_{k\in I_m}\log|\theta_{m,k}^{f,\pi}| = \infty$.
	\\ This implies from \eqref{eq.alpha.estimator} that  $\Big(\alpha_0(f)- \frac{1}{2} \Big) \log \vert \pi^{m+1} \vert \xlongrightarrow{m \to \infty} \infty \iff \alpha_0(f)<\frac{1}{2}$. Thus, the last inequality of \eqref{eq.alpha.estimator} is equivalent to 
	\begin{align*}
		\limsup_{m} \max_k\frac{ \log \vert \theta^{f,\pi}_{m,k} \vert}{(\alpha_0(f)-\frac{1}{2})\log \vert \pi^{m+1} \vert} = 1
		\iff \frac{1}{2}-\alpha_0(f) =  \limsup_m \max_k \frac{\log|\theta^{f,\pi}_{m,k}|}{-\log|\pi^{m+1}|}.
	\end{align*}
	\item  $\limsup_m \max_{k\in I_m}\log|\theta_{m,k}^{f,\pi}| = -\infty$
	\\ Similarly, we have from \eqref{eq.alpha.estimator} that  $\Big(\alpha_0(f)- \frac{1}{2} \Big) \log \vert \pi^{m+1} \vert \xlongrightarrow{m \to \infty} -\infty \iff \alpha_0(f)>\frac{1}{2}$, thus 
	\begin{align*}
		\limsup_{m} \max_k\frac{ \log \vert \theta^{f,\pi}_{m,k} \vert}{(\alpha_0(f)-\frac{1}{2})\log \vert \pi^{m+1} \vert} = 1
		\iff \alpha_0(f)-\frac{1}{2} =  \limsup_m \max_k \frac{\log|\theta^{f,\pi}_{m,k}|}{\log|\pi^{m+1}|}.
	\end{align*}
	\item $\limsup_m \max_{k\in I_m}\log|\theta_{m,k}^{f,\pi}| = L \in (-\infty, \infty)$
	\\ This implies from \eqref{eq.alpha.estimator} that  $\Big(\alpha_0(f)- \frac{1}{2} \Big) \log \vert \pi^{m+1} \vert \xlongrightarrow{m \to \infty} \widetilde{L}$ for some constant $\widetilde{L}$. However, since $\log \vert \pi^{m+1} \vert \to -\infty$ as $m \to \infty$, the only possible option is $\alpha_0(f)=\frac{1}{2}$. 
\end{enumerate}
\end{proof}

\begin{remark} [Comparing the two estimators of H\"older regularity]
The major difference between the new characterization \eqref{holder.estimate.new} of H\"older regularity and the classical one in \eqref{eq.Holder.estimate.old} is that the new one uses the Schauder coefficient $|\theta_{m,k}^{f,\pi}|$, instead of projection $\langle f,\psi_{m,k}\rangle$ along an appropriate orthonormal basis. Along any complete refining and balanced partition sequence $\pi$, the Schauder coefficient $|\theta_{m,k}^{f,\pi}|$ can be precisely computed (without any approximation) via equation \eqref{eq.theta.coeff} unlike $\langle f,\psi_{m,k}\rangle$ which is defined as $\int_{\mathbb{R}}f(u)\psi_{m,k}(u) du$. Therefore, the new characterization will have less approximation error as an estimator for H\"older regularity.
\end{remark}

\vspace{-7mm}
\begin{figure}[h!]
\centering
\subfloat[Brownian motion]{\includegraphics[width=.49\textwidth]{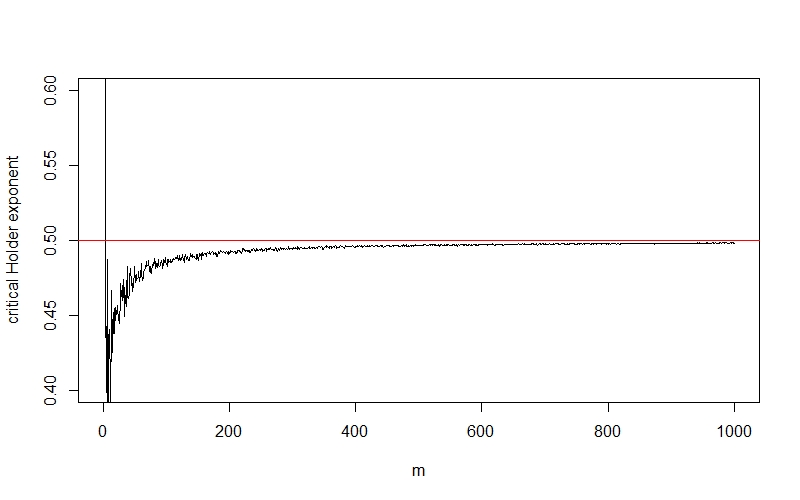}}
\subfloat[Fake Brownian motion]{\includegraphics[width=.49\textwidth]{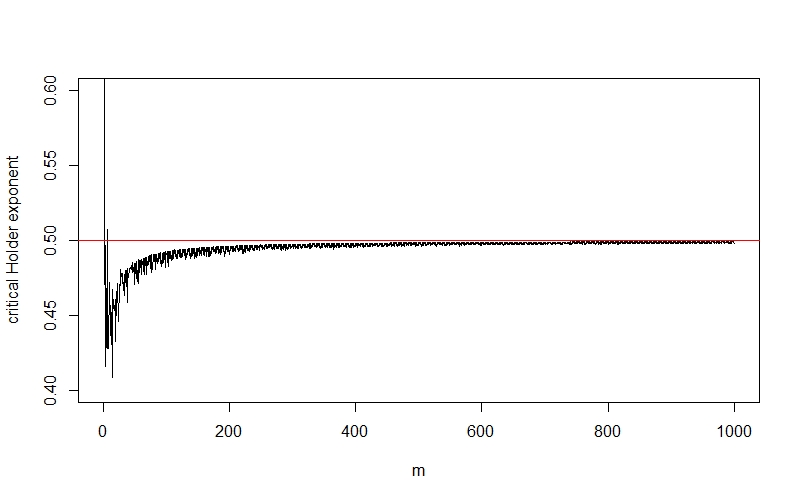}}
\caption{Estimating H\"older regularity}
\label{fig: estimator}
\end{figure}
\vspace{-5mm}

In Figure \ref{fig: estimator}, we plot the quantities $\max_k \frac{\log|\theta_{m,k}|}{\log|\pi^{m+1}|}$ for different values of $m$, where $(\theta_{m,k})$ are respective Schauder coefficients of Brownian motion and its fake version along the dyadic partition sequence on $[0,1]$ (for more details on the process faking BM, see Example \ref{ex. mixed coefficients}). We leave further properties of the new H\"older estimator based on \eqref{holder.estimate.new}, such as convergence rate, as future work.

\section{Critical H\"older exponent and variation index}    \label{sec. critical Holder and variation index}

Because of the inclusion $C^{\beta}([0, T], \mathbb{R}) \subset C^{\alpha}([0, T], \mathbb{R})$ for any $0 < \alpha \le \beta \le 1$, we introduce the following concept of critical H\"older exponent.

\begin{definition} [Critical H\"older exponent]\label{def: holder exponent}
We define \textit{critical H\"older exponent} of $x \in C^0([0, T], \mathbb{R})$
\begin{equation}    \label{def.critical Holder exponent}
    c(x, [0, T]) := \sup \Big\{ \alpha \in (0, 1] : x\in C^\alpha([0, T], \mathbb{R}) \Big\}.
\end{equation}
\end{definition}

\begin{definition}[Variation index along a partition sequence \cite{das2022theory}]  \label{def. variation index}
The variation index of $x\in C^0([0,T],\mathbb{R})$ along a refining partition sequence $\pi$ with vanishing mesh is defined as
\begin{equation}    \label{def.variation index}
    p^\pi(x) := \inf \bigg\{p\geq 1 \, : \, \limsup_{n\uparrow \infty} \sum_{t^n_i, t^n_{i+1} \in \pi^n} \big\vert x(t^n_{i+1})-x(t^n_{i})\big\vert^p < \infty \bigg\}.
\end{equation}
\end{definition}

We note that the rough path theory defines the \textit{$p$-variation of $x$} by taking the supremum of $p$ variations over all partitions of $[0, T]$ for a given continuous function $x \in C^0([0, T], \mathbb{R})$:
\begin{equation*}
\Vert x \Vert_{p-var} := \bigg(\sup_{\pi\in \Pi([0,T])} \sum_{t_j, t_{j+1} \in \pi} \big\vert x(t_{j + 1}) - x(t_j) \big\vert^p \bigg)^{1/p},
\end{equation*}
where $\Pi([0,T])$ denotes the set of all partitions of $[0,T]$. Also, it is well-known that $\Vert x \Vert_{\frac{1}{\alpha}-var} < \infty$ holds for any $x \in C^{\alpha}([0, T], \mathbb{R})$. We formalize this in the following lemma with an inequality involving the variation index and critical H\"older exponent of a continuous function. Its proof is provided in Supplement A \cite{suppA}.

\begin{lemma}   \label{lem. variation index inequality}
For any $x \in C^{\alpha}([0, T], \mathbb{R})$, $x$ has finite $\frac{1}{\alpha}$-variation, i.e., $\Vert x \Vert_{\frac{1}{\alpha}-var} < \infty$. Furthermore, for every refining sequence $\pi$ of partitions on $[0, T]$ with vanishing mesh, we have the inequality
\begin{equation}    \label{ineq. variation index and critical exponent}
	\frac{1}{p^\pi(x)}\geq c(x,[0,T]).
\end{equation}
\end{lemma}

Even though the critical H\"older exponent and the reciprocal of variation index are the same for fractional Brownian motions along any partition sequence with mesh size converging to zero fast enough (as a result of self-similarity), they are not the same for general processes. However, in some financial applications, when estimating (H\"older) roughness of a path, one estimates $p$-th variation instead of H\"older regularity. The following provides an example of a deterministic function with critical H\"older exponent different from the reciprocal of variation index (i.e., the inequality \eqref{ineq. variation index and critical exponent} is strict). Therefore, one should not determine H\"older regularity by measuring the variation index without appropriate normality assumption on the path.

\begin{example}[A function with critical H\"older exponent different from reciprocal of variation index]   \label{ex.different holder and variation index}
For a fixed $0 < \epsilon_0 < \frac{1}{3}$, consider a function $x : [0,1] \rightarrow \mathbb{R}$ given by the Schauder representation along the dyadic partition $\mathbb{T}$
\begin{equation}    \label{eq.different holder and variation index}
	x(t) = \sum_{m=0}^{\infty} \sum_{k \in I_m} \theta_{m,k} e^{\mathbb{T}}_{m,k}(t),
\end{equation}
where the Schauder coefficients are given by
\begin{align*}
	\theta_{m,k} = \begin{cases}
		2^{\epsilon_0 m}, & k = 0, \, \left\lfloor\frac{2^m}{\sqrt[4]{m}}\right\rfloor\wedge (2^m-1), \, 2\left\lfloor\frac{2^m}{\sqrt[4]{m}}\right\rfloor\wedge (2^m-1), \cdots \\
		0, & \text{otherwise}.
	\end{cases}
\end{align*}
Then, the function $x$ is continuous and we have the strict inequality
\begin{equation*}
	\frac{1}{p^{\mathbb{T}}(x)} > c(x, [0,1]).
\end{equation*}
\end{example}

\begin{proof}
Lemma~\ref{lem: continuity condition} (iii) implies that $x$ is continuous. We will show that $c(x,[0,T]) = \frac{1}{2}-\epsilon_0$. It is easy to verify
\begin{equation*}
	\sup_{m,k}\big( |\theta_{m,k}|\times 2^{m(\alpha-\frac{1}{2})}\big) = \sup_{m,k}\big( 2^{(\epsilon_0+\alpha-\frac{1}{2})m}\big) =
	\begin{cases}
		0, \qquad &\alpha < \frac{1}{2}-\epsilon_0,
		\\
		1, \qquad &\alpha=\frac{1}{2}-\epsilon_0,
		\\
		\infty, \qquad &\alpha>\frac{1}{2}-\epsilon_0,
	\end{cases}
\end{equation*}
thus, Theorem \ref{main.thm} implies $x\in C^\alpha([0,1],\mathbb{R})$ for every $\alpha\leq \frac{1}{2}-\epsilon_0$. From the definition of critical exponent, this implies $c(x,[0,1]) = \frac{1}{2}-\epsilon_0$.
We now show that $x$ has finite quadratic variation along $\mathbb{T}$. Quadratic variation at level $n$ can be represented with the Schauder coefficients from Proposition 4.1 of \cite{das2021}:
\begin{align*}
	[x]^{(2)}_{\mathbb{T}^n}(1) &= \frac{1}{2^n} \sum_{m=0}^{n-1}\sum_{k=0}^{2^m-1} \theta_{m,k}^2=\frac{1}{2^n}\sum_{m=0}^{n-1} 2^{2\epsilon_0 m} \times \lfloor \sqrt[4]{m}\rfloor
	\leq \frac{1}{2^n} \sum_{m=0}^{n-1} 2^{2\epsilon_0m + \frac{1}{4}log_2 m}\\
	& \leq \frac{1}{2^n}\sum_{m=0}^{n-1} 2^{(2\epsilon_0  + \frac{1}{4})m} = \frac{1}{2^n} \cdot \frac{2^{(2\epsilon_0  + \frac{1}{4})n}-1}{2^{2\epsilon_0  + \frac{1}{4}}-1} \xrightarrow[]{n\to \infty} 0.
\end{align*}  
The last limit follows from the fact that $\epsilon_0<\frac{1}{3}$ implies $(2\epsilon_0  + \frac{1}{4})-1<0$. Therefore, we have $x\in Q_\mathbb{T}([0,1],\mathbb{R})$ with $[x]^{(2)}_\mathbb{T}(1) = 0$. From the definition of the variation index, we have $p^\mathbb{T}(x)\leq 2$ which implies $\frac{1}{p^\mathbb{T}(x)} \geq \frac{1}{2}>\frac{1}{2}-\epsilon_0 = c(x,[0,1])$. This concludes the proof.
\end{proof}

The $p$-th variation of (a sample path of) a stochastic process along a given sequence of partitions and H\"older exponent are often used as an important tool to identify the Hurst index of a fractional Gaussian process. The following example shows one can construct a deterministic function which has the same `roughness' order as Brownian motion but does not possess a Gaussian structure.

\begin{example}[A function having the same path properties as Brownian motion]\label{ex.sqrt(m)}
We consider the dyadic partition $\mathbb{T}$ of $[0, 1]$ and $x :[0, 1] \rightarrow \mathbb{R}$ with the Schauder representation
\begin{align}    \label{eq.sqrt(m)}
    x(t) &= \sum_{m=0}^{\infty} \sum_{k \in I_m} \theta_{m,k} e^{\mathbb{T}}_{m,k}(t), \quad \text{where} \quad
    \theta_{m,k} = 
        \begin{cases}
		\sqrt{m}, & k = 0, m, 2m, \cdots, \left\lfloor \frac{2^m-1}{m}\right\rfloor m,
		\\
		~~ 0, & \text{otherwise}.
        \end{cases}
\end{align}
Then, $x \in C^0([0, 1], \mathbb{R})$, $\Vert x \Vert_{C^{1/2}([0, 1])} = \infty$, and its quadratic variation along $\mathbb{T}$ is non-trivial and linear, i.e., $[x]_{\mathbb{T}}(t) := \limsup_{n \to \infty} \sum_{t^n_i, t^n_{i+1} \in \mathbb{T}^n} \vert x(t^n_{i+1})-x(t^n_{i})\vert^2 = t$ for any $t \in [0, 1]$. Almost every path of standard Brownian motion has these properties.
\end{example}
\begin{proof}

The continuity of $x$ is straightforward from Lemma~\ref{lem: continuity condition} (iii).

Since we have $\sup_{m,k}(|\theta_{m,k}| |\pi^{m+1}|^{\frac{1}{2}-\alpha}) = \sup_{m} (\sqrt{m}\times2^{(m+1)(\alpha-\frac{1}{2})})$, this quantity is infinity when $\alpha=\frac{1}{2}$. Hence, Theorem \ref{main.thm} concludes $\Vert x \Vert_{C^{1/2}([0, 1])} = \infty$. A similar argument shows that $\sup_{m,k} (|\theta_{m,k}| |\pi^{m+1}|^{\frac{1}{2}-\alpha})<\infty$, hence $\Vert x\Vert_{C^{\alpha}([0, 1])} < \infty$, if $\alpha<\frac{1}{2}$.
	
We now show that the quadratic variation of $x$ along $\mathbb{T}$ exists and is strictly positive. For a fixed $t \in (0, 1]$, the quadratic variation up to level $n$ can be represented in terms of the Schauder coefficients from Proposition 4.1 of \cite{das2021}:
\begin{align}    \label{eq.int.part}
    [x]^{(2)}_{\mathbb{T}^n}(t) = \frac{1}{2^n} \sum_{m=0}^{n-1}\sum_{k=0}^{2^m-1} \theta_{m,k}^2 \mathbbm{1}_{\frac{k}{2^m}<t} 
    = \frac{1}{2^n} \sum_{m=0}^{n-1}\sum_{k=0}^{\lceil t2^m\rceil-1} \theta_{m,k}^2    
    = \frac{1}{2^n} \sum_{m=0}^{n-1} m \times \frac{\lceil t2^m\rceil }{2^m}\Big(\lfloor \frac{2^m-1}{m}\rfloor+1 \Big).
\end{align}
Using the equalities $2^m-m-1 < m \lfloor \frac{2^m-1}{m}\rfloor \leq 2^m-1$ and $t\leq \frac{\lceil t2^m\rceil }{2^m} < \frac{2^mt+1}{2^m}$, we obtain from \eqref{eq.int.part} the following two-sided bounds on $[x]^{(2)}_{\mathbb{T}^n}(t)$
\begin{align*}
    \frac{1}{2^n} \sum_{m=0}^{n-1} t(2^m-1) < &[x]^{(2)}_{\mathbb{T}^n}(t)
    < \frac{1}{2^n} \sum_{m=0}^{n-1} (t+\frac{1}{2^m})(2^m+m-1)
    \\
    \Rightarrow  t\frac{1}{2^n} (2^n-1 - n) < &[x]^{(2)}_{\mathbb{T}^n}(t)< t\frac{1}{2^n} \Big( 2^n-1 + \frac{(n-1)n}{2}-n \Big)+ \frac{2n}{2^n}.
\end{align*}
Taking limit $n \to \infty$, we conclude $[x]^{(2)}_\mathbb{T}(t) = t$, i.e., $x$ has non-trivial quadratic variation along $\mathbb{T}$.
\end{proof}

\begin{remark}  \label{rem.make random}
Even though we construct deterministic functions with certain path properties in Example \ref{ex.different holder and variation index} and \ref{ex.sqrt(m)}, we can make them stochastic processes with random Schauder coefficients. For example, consider a sequence of Bernoulli random variables $\{\beta_{m, k}\}$ with $\beta_{m, k} \sim Bernoulli(1/2)$ (these random variables $\beta_{m,k}$ do not need to be independent), replace the deterministic coefficients $\theta_{m, k}$ with the random coefficients $\theta_{m, k}\beta_{m, k}$ in the representations \eqref{eq.different holder and variation index} and \eqref{eq.sqrt(m)}, respectively, and denote the resulting process by $Y$. Then, the above arguments can be applied to show that almost every path of the stochastic process $Y$ has the same properties as in Example \ref{ex.different holder and variation index} and \ref{ex.sqrt(m)}, respectively.
\end{remark}

Figure \ref{fig:sqrt_m} illustrates a deterministic function $t \mapsto x(t)$ in Example \ref{ex.sqrt(m)} and a sample path of the corresponding process $Y$ mentioned in Remark \ref{rem.make random}, together with their quadratic variations along the dyadic partitions, truncated at level $n=16$.

\vspace{-7mm}
\begin{figure}[!htb]
\centering
\subfloat[Function $x$ and a sample path of $Y$]		{\includegraphics[width=.49\linewidth]{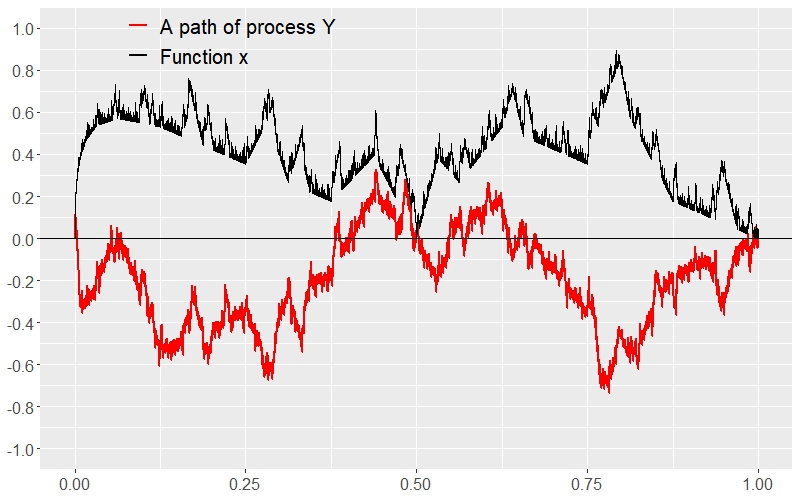}}
\subfloat[Quadratic variations]{\includegraphics[width=.49\linewidth]{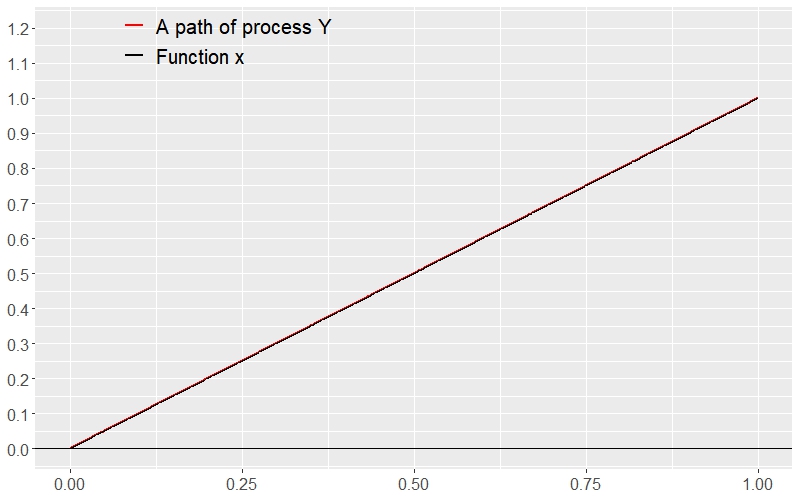}}
\caption{Function $x$, a sample path of $Y$, and their quadratic variations in Example \ref{ex.sqrt(m)} and Remark  \ref{rem.make random}.}
\label{fig:sqrt_m}
\end{figure}
\vspace{-5mm}

\section{Fake Brownian motion} \label{sec. mimicking BM}

In this section, we construct stochastic processes which fake Brownian motion using the Schauder representation. In what follows, we assume $T=1$ without loss of generality, as the theory can be easily extended to any finite interval. For a finitely refining sequence $\pi$ of partitions of $[0, 1]$, we have the Schauder representation of a Brownian motion $B$ along $\pi$, from Proposition \ref{prop:coeff_hat_func} (more precisely Lemma 5.1 of \cite{das2021}):
\begin{equation}    \label{eq. standard normal coefficients}
B(t) = Zt + \sum_{m=0}^{\infty} \sum_{k \in I_m} \theta_{m, k}^{B} e^{\pi}_{m, k}(t), \qquad t \in [0, 1],
\end{equation}
where $Z = B(1)$ and the coefficients $\theta^B_{m, k}$ are independent and identically distributed standard normal random variables. 

We shall now construct \textit{non-Gaussian} processes, which have many properties in common with Brownian motion by replacing the normal Schauder coefficients $\theta^B_{m, k}$ with some non-normal random variables. Furthermore, we shall show that these non-Gaussian processes even `mimic' the normality; many classical normality tests conclude that these fake processes have Gaussian marginals despite being theoretically non-Gaussian.

\subsection{Construction of fake Brownian processes}    \label{subsec. mimicking BM}
We shall focus on the construction of fake processes on the interval $[0,1]$, as Remark \ref{parching.BM} below explains how the construction method can be extended to the entire non-negative real line $[0, \infty)$. The following theorem provides a list of properties \textit{(BM i - vii)} of Brownian motion $B$ in \eqref{eq. standard normal coefficients} we shall fake.

\begin{theorem}[Construction of fake Brownian motion]    \label{thm.fake-BM}
For a balanced and complete refining sequence of partition $\pi$ of $[0,1]$, consider the stochastic process $Y$ defined on $[0, 1]$ via Schauder representation along $\pi$
\begin{equation*}
	Y(t) =Zt+ \sum_{m=0}^{\infty} \sum_{k \in I_m} \theta^Y_{m,k} e^{\pi}_{m,k}(t),
\end{equation*}
where $\{\theta_{m,k}^Y\}$ are independent random variables satisfying
\begin{enumerate} [label=(\alph*)]
\item $\mathbb{E}[\theta^Y_{m,k}] = 0$ and $Var[\theta^Y_{m,k}]=1$,
\item there exist positive constants $C_1$ and $C_2$ such that 
\begin{equation*}
\mathbb{P} \big[ |\theta^Y_{m, k}| \ge \delta \big] \le C_1 \exp \bigg(-\frac{{\delta}^2}{2C_2} \bigg)
\end{equation*}
holds for every $\delta > 0$, $k \in I_m$, $m \ge 0$.
\end{enumerate}
Then, the process $Y$ satisfies the following properties (BM i-vii).
\begin{enumerate}
	\item[(BM i)] $Y$ has a continuous path almost surely.
	\item[(BM ii)] $\mathbb{E}[Y(t)] = 0$ and $\mathbb{E}[Y(t)Y(s)] = t \wedge s$ for any $t, s \in [0, 1]$.
	\item[(BM iii)] Increments of $Y$ over disjoint intervals are uncorrelated, i.e., $\mathbb{E}[(Y(t)-Y(s))(Y(v)-Y(u))] = 0$ for every $0 \le s \le t \le u \le v \le 1$.
	\item[(BM iv)] Almost every path of $Y$ belongs to $C^{\frac{1}{2}-}([0, T], \mathbb{R})$. 
	\item[(BM v)] The Schauder coefficients $\theta_{m, k}^{Y}$ have mean $0$, variance $1$, and are uncorrelated, i.e., $\mathbb{E}(\theta_{m, k}^{Y}) = 0, \text{Var}(\theta_{m, k}^{Y}) = 1$ and $\mathbb{E}(\theta_{m, k}^{Y}\theta_{m', k}^{Y}) = \mathbbm{1}_{m=m'}\mathbbm{1}_{k=k'}$.
\end{enumerate}
Moreover, in addition to (a) and (b), if we further assume the condition
\begin{enumerate}
	\item [(c)] $\mathbb{E}[(\theta^Y_{m,k})^4] < \infty$ for every $k \in I_m$, $m \ge 0$,
\end{enumerate}
then, the process $Y$ satisfies the following additional properties
\begin{enumerate} [label=(BM \roman*)]
	\setcounter{enumi}{6}
	\item The quadratic variation along $\pi$ is given as $[Y]_{\pi}(t) = t$ for every $t \in [0, 1]$ almost surely.
	\item  The quadratic variation along any coarsening partition $\nu$ of $\pi$ (see Definition 6.1 of \cite{das2021} for further details on coarsening) is given as $[Y]_{\nu}(t) = t$ for every $t \in [0, 1]$ almost surely.
\end{enumerate}
\end{theorem}

\begin{remark}  \label{rem. critical Holder BM}
The property \textit{(BM iv)} can be strengthened under an additional assumption on the Schauder coefficients. 
If all random variables $(\theta^Y_{m, k})$ have the same bounded support such that
\begin{equation*}
	\mathbb{P} \Big[ \sup_{m, k} |\theta^{Y}_{m, k}(\omega)| \le C \Big] = 1
\end{equation*}
holds for some constant $C > 0$, then $Y$ of Theorem \ref{thm.fake-BM} satisfies

\noindent (BM iv') \textit{Almost every path of $Y$ has the critical H\"older exponent $1/2$.}
\end{remark}

\begin{proof}
We first note that we use some results (Theorem \ref{thm. finite moments} and Proposition \ref{prop. finiteness condition}) of Section \ref{sec. moment} in this proof. We shall prove each property from \textit{(BM i)} to \textit{(BM vii)} one by one.

\textbf{(BM i)} Since condition (b) coincides with that of Lemma \ref{lem: continuity condition} with $\epsilon = 1/2$, the result follows from Lemma \ref{lem: continuity condition}.

\textbf{(BM ii)} First, Proposition \ref{prop. finiteness condition} for the case $H = 1/2$ shows that the family of coefficients $(\theta^B_{m, k})$ satisfies the condition \eqref{ineq. moment bounds} for $\ell = 2$. Moreover, we have the uniform bound
\begin{equation*}
	\mathbb{E} \big| \theta^Y_{m_1, k_1} \theta^Y_{m_2, k_2} \big| \le \frac{1}{2} \Big( \mathbb{E}[(\theta^Y_{m, k})^2] + \mathbb{E}[(\theta^Y_{m', k'})^2] \Big) = \mathbb{E}[(\theta^Y_{m, k})^2] \le \sqrt{\mathbb{E}[(\theta^Y_{m, k})^4]} < \infty,
\end{equation*}
and the argument in the proof of Proposition \ref{prop. finiteness condition} yields that the condition \eqref{ineq. moment bounds} for $\ell = 2$ is satisfied for the family $(\theta^Y_{m, k})$ as well. Theorem~\ref{thm. finite moments} then proves the property, since we matched the mean and covariance of the Schauder coefficients, i.e., $\mathbb{E}[\theta^Y_{m, k}] = \mathbb{E}[\theta^B_{m, k}] = 0$, $\mathbb{E}[\theta^Y_{m, k}\theta^Y_{m', k'}] = \mathbb{E}[\theta^B_{m, k}\theta^B_{m', k'}] = \mathbbm{1}_{m=m'}\mathbbm{1}_{k=k'}$.

\textbf{(BM iii)} From \textit{(BM ii)}, we have 
\begin{align*}
	 \quad \mathbb{E}[(Y(t)-Y(s))(Y(v)-Y(u))] 
	&= \mathbb{E}[Y(t)Y(v)] -\mathbb{E} [Y(t)Y(u)] - \mathbb{E}[Y(s)Y(v)] + \mathbb{E} [Y(s)Y(u)]
	\\
	&= t-t-s+s = 0.
\end{align*}
We note that Brownian motion has independent increments, however, uncorrelatedness between increments does not generally imply independent increments for non-Gaussian process $Y$.

\textbf{(BM iv)} This follows from Corollary \ref{cor. Holder continuity}.

\textbf{(BM iv') of Remark \ref{rem. critical Holder BM}} The quantity $\sup_{m,k}(|\theta^Y_{m,k}||\pi^{m+1}|^{\frac{1}{2}-\alpha})$ is finite if and only if $\alpha \le 1/2$, thus Theorem \ref{main.thm} yields the result.

\textbf{(BM v)} This is obvious from the construction of $Y$.

\textbf{(BM vi)} The process $Y$ belongs to the class $\mathcal{B}^{\pi}$ in (13) of \cite{das2021}, thus $[Y]_{\pi}(t) = t$ almost surely from \cite[Theorem ~5.2]{das2021}.

\textbf{(BM vii)} The process $Y$ further belongs to the class $\mathcal{A}^{\pi}$ in \cite[Equation (31)]{das2021}, thus Theorem 6.5 of \cite{das2021} proves the property.
\end{proof}
We note that the fake process $Y$ in Theorem \ref{thm.fake-BM} cannot be a martingale (for any filtration) unless Brownian motion. If it were a martingale, then it is continuous from \textit{(BM i)}, square-integrable, i.e., $\mathbb{E}[(Y(t))^2] = t < \infty$ from \textit{(BM ii)}, and  \cite[Theorem ~1.5.13]{KS1} yields the existence of a unique (up to indistinguishability) continuous process $\langle Y \rangle_t$ such that $Y^2_t - \langle Y \rangle_t$ is a martingale. Moreover,  \cite[Theorem ~1.5.8]{KS1} and \textit{(BM vi)} conclude that $\langle Y \rangle_t = [Y]_{\pi}(t) = t$ holds almost surely for every $t \in [0, 1]$. The continuity of $t \mapsto \langle Y \rangle_t$ yields that the process $\langle Y \rangle_t$ is indistinguishable from $t$ and L\'evy's characterization of Brownian motion concludes that $Y$ should be a Brownian motion, which is a contradiction.

\begin{remark}[Fake Brownian motion on $[0,\infty)$]   \label{parching.BM}
Our construction method of fake Brownian motion in this subsection can be extended to the interval $[0, \infty)$ as Brownian motion has independent increments. Using the construction, we can define a sequence of independent fake Brownian motions $Y^{(i)}$ on intervals $[i-1, i]$ for $i = 1, 2, \cdots $ and patch them recursively by 
\begin{align*}
	Y_t := Y^{(1)}(t), \qquad &0 \le t \le 1,
	\\
	Y_t := Y_n + Y^{(n+1)}_{t-n}, \qquad &n \le t \le n+1.
\end{align*}
Then, the resulting process $Y$ has the same aforementioned properties as Brownian motion on $[0,\infty)$. See \cite[Corollary ~2.3.4]{KS1} for further details on this patching. However, this patching method does not work for fake fractional Brownian motions in the next subsection, since the increments of fBM are correlated for all Hurst index $H\neq 1/2$.
\end{remark}

\subsection{Examples} \label{subsec. examples BM}
We now provide some examples of fake processes that satisfy the assumptions of Theorem \ref{thm.fake-BM}.

\begin{example} [Uniformly distributed coefficients]\label{ex. uniform coefficients}
Let us denote $U$ a uniformly distributed random variable over the interval $[-\sqrt{3}, \, \sqrt{3}]$, i.e. $U \sim \text{Uniform}(-\sqrt{3}, \, \sqrt{3})$ such that its mean and variance are zero and one respectively. For a given complete refining and balanced sequence $\pi$ of partitions of $[0, 1]$, consider a process $Y$ on $[0, 1]$ with the representation
\begin{equation}    \label{eq. uniform coefficients}
    Y(t) = Zt+ \sum_{m=0}^{\infty} \sum_{k \in I_m} \theta^Y_{m,k} e^{\pi}_{m,k}(t),
\end{equation}
where the Schauder coefficients $\theta^Y_{m, k}$ are independent and identically distributed as $U$ and $Z\sim N(0,1)$ independent of all the other Schauder coefficients. Then, from Theorem \ref{thm.fake-BM} the process $Y$ has properties \textit{(BM i - vii)} as Brownian motion.
\end{example}

\begin{example} [Centered, scaled beta distributed coefficients]    \label{ex. beta coefficients}
Since the beta distribution is a family of continuous probability distributions with compact support which generalizes uniform distribution, we can use beta distribution as a choice of Schauder coefficients. For example, we consider two-centered, scaled beta distributed random variables
\begin{equation}   \label{def.beta.dist}
    B_1 \sim \sqrt{20}\Big(\text{Beta}(2, 2) - \frac{1}{2}\Big), \qquad B_2 \sim \sqrt{8}\Big(\text{Beta}(\frac{1}{2}, \frac{1}{2}) - \frac{1}{2}\Big),
\end{equation}
such that both of them have mean zero and variance one. We construct two processes $Y_1$ and $Y_2$ on $[0, 1]$ such that their Schauder coefficients $\theta^{Y_i}_{m, k}$ in the representation \eqref{eq. uniform coefficients} are independent and identically distributed as $B_i$ of \eqref{def.beta.dist} for $i = 1, 2$.
\end{example}

Figure 3 of Supplement B \cite{suppB}  provides simulations of the processes $B(t)-Zt$ of \eqref{eq. standard normal coefficients}, $Y-Zt$ in Example \ref{ex. uniform coefficients}, and $Y_i - Zt$ in Example \ref{ex. beta coefficients} for $i = 1, 2$, in black, red, blue, and green lines respectively, along the dyadic partitions on $[0, 1]$, up to level $n = 15$. Note that these trajectories have both initial and terminal values to be $0$. To have a trajectory of a real Brownian motion (and its fake versions), we can add a linear term $Zt$ for an independent standard normal random variable $Z$, to those sample paths in Figure 3(a). Figure 3(b) shows that all paths indeed have the same quadratic variation as standard Brownian motion.

\begin{example} [Mixed coefficients]    \label{ex. mixed coefficients}
In order to construct fake Brownian motion, one can also mix standard normal random variables and other variables with the same mean and variance when choosing Schauder coefficients. For example, given the Schauder representation \eqref{eq. uniform coefficients}, we set the coefficients to be all independent, and distributed as
\begin{equation}    \label{eq.mixed coefficients}
	\theta^{Y_{mix}}_{m, k} \sim 
	\begin{cases}
		\qquad N(0, 1), \qquad &\text{if } m \text{ is odd},
		\\
		\text{Uniform}(-\sqrt{3}, \, \sqrt{3}), \qquad &\text{if } m \text{ is even}.
	\end{cases}
\end{equation}
\end{example}
Figure 4 of Supplement B \cite{suppB} describes this fake process and Section \ref{subsec. normality BM} provides normality test results for the fake process.

\begin{example} [Fake Brownian processes along a non-uniform partition]    \label{ex. non-dyadic}
Since the fake processes in Examples \ref{ex. uniform coefficients} - \ref{ex. mixed coefficients} were constructed along the same dyadic partitions, we provide an example which is contracted along a `non-uniform' partition sequence. Define a sequence of partitions $\pi =(\pi^n)_{n \in \mathbb{N}}$ with $\pi^n =\big(0=t^n_1 < \cdots <t^n_{N(\pi_n)} = 1)$ satisfying for each $n \in \mathbb{N}$
\begin{equation*}
	t^{n+1}_{2k} = t^n_k, \qquad t^{n+1}_{2k+1} = t^n_k + \frac{t^n_{k+1} -t^n_k}{2.5}, \qquad \forall \, k = 1, \cdots, 2^n.
\end{equation*}
\end{example}
Figure 5 of Supplement B \cite{suppB}  provides sample paths and corresponding quadratic variations of the fake processes in Example \ref{ex. uniform coefficients} - \ref{ex. mixed coefficients} along the non-uniform partition sequence in Example \ref{ex. non-dyadic}.

\subsection{Normality tests on the marginal distributions: theory vs reality}\label{subsec. normality BM} 

Since we replaced normally distributed Schauder coefficients of Brownian motion with other non-Gaussian coefficients, the fake processes we constructed in the previous subsection are non-Gaussian processes, even with the mixed coefficients in Example \ref{ex. mixed coefficients}. In particular, the finite-dimensional distributions of these processes cannot be identified as any well-known distribution, even though the first two moments coincide with those of Brownian motion (property \textit{(BM ii)} of Theorem \ref{thm.fake-BM}). Nonetheless, it turns out that those non-Gaussian processes we constructed in the previous subsection can even mimic the `Gaussian' property.

In this regard, we first simulated $5000$ sample paths of the fake Brownian process $Y_{mix}$ in Example \ref{ex. mixed coefficients}, defined on the support $[0, 1]$ along the dyadic partitions, truncated at level $n = 15$. We randomly chose $10$ dyadic points $t_i \in (0, 1)$ (such that each $t_i$ is given as $k/2^{15}$ with $k \in \mathbb{Z}$) and took $5000$ sample values $Y_{mix}(t_i)$ at each $t_i$ for $i = 1, 2,\cdots, 10$. Then, we did three commonly used normality tests, namely Shapiro–Wilk test, Kolmogorov–Smirnov test, and Jarque–Bera test, to decide whether those $5000$ sample values for each point $t_i$ are drawn from a Gaussian distribution.

Table 1 of Supplement B \cite{suppB}  provides the $p$-values of the three normality tests for $10$ randomly chosen dyadic points. As we can observe, most of the $p$-values are bigger than $5\%$, meaning we should not reject the null hypothesis (with a significance level of $0.05$) stating that our sample values are drawn from a Gaussian distribution; the $p$-values which are less than $5\%$ are displayed in boldface. Figure 6 of Supplement B \cite{suppB} also gives histograms and Q-Q plots of the $5000$ sample values of $Y_{mix}(t)$ at another two (randomly chosen dyadic) points, $t = 102/2^{15}$ and $6653/2^{15}$.

Despite the finite-dimensional distributions of the fake Brownian process $Y_{mix}$ are not \textit{theoretically} Gaussian, the above simulation study concludes that they cannot be \textit{statistically} distinguished from a Gaussian marginal. In practice, when observing entire sample trajectories (or some discrete sample signals) drawn from an unknown process, we can check some path properties such as \textit{(BM i, iv, vi)} and compute moments \textit{(BM ii, iii)} to find out the distribution of the unknown process. However, our construction method suggests that it would be difficult to distinguish real Brownian motion from those non-Gaussian fake processes.

\begin{remark} [Faking higher order moments] \label{rem. third, fourth moment}
When constructing the fake fractional processes in Theorem \ref{thm.fake-BM}, we only matched the first two moments of the Schauder coefficients as those of Brownian motion. Nevertheless, most of the $p$-values in Table 1 of Supplement B \cite{suppB} are quite big for the Jacque-Bera test, which is a goodness-of-fit test of whether sample data have the third and fourth moments (skewness and kurtosis) matching a normal distribution. If we want to further mimic those higher-order moments, we may choose Schauder coefficients as random variables with the same moments up to higher-order as the standard normal distribution,  due to Theorem \ref{thm. finite moments} later. For example, a discrete random variable
\begin{equation*}
	R \sim
	\begin{cases}
		&~~\sqrt{3}, \qquad \text{with probability } 1/6,
		\\
		&-\sqrt{3}, \qquad \text{with probability } 1/6,
		\\
		&~~~ 0, ~~ \qquad \text{with probability } 2/3,
	\end{cases}        
\end{equation*}
has the same first four moments with the standard normal distribution and satisfies the conditions of Theorem \ref{thm.fake-BM}.
\end{remark}

\section{Fake fractional Brownian motion} \label{sec. mimicking fBM}

A fractional Brownian motion~(fBM) $B^H$, defined on $[0, 1]$, with Hurst index $H \in (0, 1)$ has the following Schauder representation along a finitely refining partition sequence $\pi$ from Proposition \ref{prop:coeff_hat_func}:
\begin{equation}    \label{eq. non-standard normal coefficients}
B^H(t) = Zt+ \sum_{m=0}^{\infty} \sum_{k \in I_m} \theta^{B^H}_{m, k} e^{\pi}_{m, k}(t), \qquad t \in [0, 1],
\end{equation}
where $Z = B^H(1)$ and $\theta^{B^H}_{m, k}$ are normally distributed with mean zeros and variance equal to $1$ and $\sigma^2_{m, k}$, respectively, with the notation
{\footnotesize
\begin{equation}    \label{eq.fbm.theta.var}
    \sigma^2_{m, k} := \frac{(\dII)^2 (\dI)^{2H} + (\dI)^2 (\dII)^{2H} - \dI \dII \Big[ (\dI+\dII)^{2H}-(\dI)^{2H}-(\dII)^{2H} \Big] }{\dI \dII (\dI + \dII)}.
\end{equation}}
Here and in what follows, we shall denote $$\dI := t^{m, k}_2 - t^{m, k}_1 \text{ and } \dII := t^{m, k}_3 - t^{m, k}_2,$$ where the support of $e^{\pi}_{m, k}$ is $[t^{m, k}_1, t^{m, k}_3]$ and its maximum is attained at $t^{k, n}_2$. We shall also use the notations $B^H\{\dI\} := B^H(t^{m, k}_2)-B^H(t^{m, k}_1)$ and $B^H\{\dII\} := B^H(t^{m, k}_3)-B^H(t^{m, k}_2)$.

Since the increments of fBM are not independent (except for the case $H = 1/2$), as 
a result the coefficients $\theta^{B^H}_{m, k}$ are not independent. Again from the expression \eqref{eq.theta.coeff}, we can compute the covariance between two Schauder coefficients:
\begin{align}
&\sigma_{m, k, m', k'} := \mathbb{E} [\theta^{B^H}_{m, k} \theta^{B^H}_{m', k'}] \label{eq.fbm.theta.cov}
\\
&= \frac{\dII \dIIp \xi^{m, k, m', k'}_{1, 1} + \dI \dIp \xi^{m, k, m', k'}_{2, 2} - \dI \dIIp \xi^{m, k, m', k'}_{2, 1}- \dIp \dII \xi^{m, k, m', k'}_{1, 2}}{\sqrt{\dI \dII \dIp \dIIp (\dI + \dII)(\dIp + \dIIp)}}       \nonumber
\end{align}
where
\begin{align*}
\xi^{m, k, m', k'}_{1, 1} &:= \mathbb{E} \big[ B^H\{\dI\} B^H\{\dIp\} \big],
\qquad
\xi^{m, k, m', k'}_{2, 2} := \mathbb{E} \big[ B^H\{\dII\} B^H\{\dIIp\} \big],
\\
\xi^{m, k, m', k'}_{2, 1} &:= \mathbb{E} \big[ B^H\{\dII\} B^H\{\dIp\} \big],
\qquad 
\xi^{m, k, m', k'}_{1, 2} := \mathbb{E} \big[ B^H\{\dI\} B^H\{\dIIp\} \big].
\end{align*}
Now that the covariance between the increments of fBM is given as
\begin{align*}
&\quad 2\mathbb{E} \Big[ \big(B^H(t) - B^H(s)\big) \big(B^H(v) - B^H(u)\big) \Big]
\\
&= \mathbb{E} \Big[ \big(B^H(t) - B^H(u)\big)^2 + \big(B^H(s) - B^H(v)\big)^2 - \big(B^H(t) - B^H(v)\big)^2 - \big(B^H(s) - B^H(u)\big)^2 \Big]
\\
&= \vert t-u \vert^{2H} + \vert s-v \vert^{2H} - \vert t-v \vert^{2H} - \vert s-u\vert^{2H}
\end{align*}
for any $s, t, u, v \in [0, 1]$, we can derive the explicit expressions
{\footnotesize
\begin{align}
\xi^{m, k, m', k'}_{1, 1} &= \frac{1}{2} \Big( \vert t^{m, k}_1 - t^{m', k'}_2 \vert^{2H} + \vert t^{m, k}_2 - t^{m', k'}_1 \vert^{2H} - \vert t^{m, k}_1 - t^{m', k'}_1 \vert^{2H} - \vert t^{m, k}_2 - t^{m', k'}_2 \vert^{2H} \Big),  \label{def. xi}
\\
\xi^{m, k, m', k'}_{2, 2} &= \frac{1}{2} \Big( \vert t^{m, k}_2 - t^{m', k'}_3 \vert^{2H} + \vert t^{m, k}_3 - t^{m', k'}_2 \vert^{2H} - \vert t^{m, k}_2 - t^{m', k'}_2 \vert^{2H} - \vert t^{m, k}_3 - t^{m', k'}_3 \vert^{2H} \Big), \nonumber
\\
\xi^{m, k, m', k'}_{2, 1} &= \frac{1}{2} \Big( \vert t^{m, k}_2 - t^{m', k'}_2 \vert^{2H} + \vert t^{m, k}_3 - t^{m', k'}_1 \vert^{2H} - \vert t^{m, k}_2 - t^{m', k'}_1 \vert^{2H} - \vert t^{m, k}_3 - t^{m', k'}_2 \vert^{2H} \Big), \nonumber
\\
\xi^{m, k, m', k'}_{1, 2} &= \frac{1}{2} \Big( \vert t^{m, k}_1 - t^{m', k'}_3 \vert^{2H} + \vert t^{m, k}_2 - t^{m', k'}_2 \vert^{2H} - \vert t^{m, k}_1 - t^{m', k'}_2 \vert^{2H} - \vert t^{m, k}_2 - t^{m', k'}_3 \vert^{2H} \Big).	\nonumber
\end{align}}
Plugging these expressions back into \eqref{eq.fbm.theta.cov}, we obtain the expression of the covariance between Schauder coefficients $\theta^{B^H}_{m, k}$ and $\theta^{B^H}_{m', k'}$, solely in terms of the partition points $t^{m, k}_i$, $t^{m', k'}_i$ for $i = 1, 2, 3$. In particular, when $(m, k) = (m', k')$, we retrieve the variance $\sigma^2_{m, k}$ of $\theta^{B^H}_{m, k}$ as in \eqref{eq.fbm.theta.var}. 

\begin{remark}
If we denote $e^{\pi}_{-1, 0}(t)$ the linear function $t \mapsto t$ on $[0, 1]$ and $\theta^{B^H}_{-1, 0} = B^H(1)$ with convention $I_{-1} = \{0\}$, the Schauder representation of fBM in \eqref{eq. non-standard normal coefficients} can be written as 
\begin{equation*}
	B^H(t) = \sum_{m=-1}^{\infty} \sum_{k \in I_m} \theta^{B^H}_{m, k} e^{\pi}_{m, k}(t), \qquad t \in [0, 1].
\end{equation*}
Then, the covariances between $Z = B^H(1)$ and the other coefficients $\{\theta^{B^H}_{m, k}\}_{m \ge 0, k \in I_m}$ in \eqref{eq. non-standard normal coefficients} can be computed in the same manner as above:
\begin{align}\label{eq.fbm.theta.cov Z}
	\sigma_{-1, 0, m, k} := \mathbb{E}[Z \;\theta^{B^H}_{m, k}] = &\frac{\dII\Big( |1-t^{m, k}_1|^{2H} + |t^{m, k}_2|^{2H} - |1-t^{m, k}_2|^{2H} - |t^{m, k}_1|^{2H} \Big)}{2\sqrt{\dI \dII (\dI + \dII)}} \nonumber   
	\\
	& \quad - \frac{\dI \Big( |1-t^{m, k}_2|^{2H} + |t^{m, k}_3|^{2H} - |1-t^{m, k}_3|^{2H} - |t^{m, k}_2|^{2H} \Big) }{2\sqrt{\dI \dII (\dI + \dII)}}.  
\end{align}
However, adding or subtracting a linear term $Zt$ from a Schauder representation does not affect the roughness properties we are interested in (e.g. continuity of path, critical H\"older exponent, $(1/H)$-th variation, etc), so we shall focus on faking $B^H(t) - Zt = \sum_{m=0}^{\infty} \sum_{k \in I_m} \theta^{B^H}_{m, k} e^{\pi}_{m, k}(t)$ such that the processes start at the origin and end at the point $(1, 0)$, as we did in Section \ref{subsec. examples BM}.
\end{remark}

When $\pi$ is the dyadic partition $\mathbb{T}$, the above expressions \eqref{eq.fbm.theta.var}, \eqref{eq.fbm.theta.cov} of variance and covariance of the Schauder coefficients reduce to a relatively simpler form:
\begin{align}
\Sigma^2_{m, k} := \mathbb{E} [(\theta^{B^H, \mathbb{T}}_{m, k})^2] &= \vert 2-2^{2H-1} \vert \times 2^{(m+1)(1-2H)}, \label{eq.var.fBM}
\\
\Sigma_{m, k, m', k'} := \mathbb{E} [\theta^{B^H, \mathbb{T}}_{m, k} \theta^{B^H, \mathbb{T}}_{m', k'}] &= 2^{\frac{1}{2}(m+m')} \big(\xi^{m, k, m', k'}_{1, 1}+\xi^{m, k, m', k'}_{2, 2}-\xi^{m, k, m', k'}_{2, 1}-\xi^{m, k, m', k'}_{1, 2} \big),    \label{eq.cov.fBM}
\end{align}
with the dyadic partition points in the definitions of $\xi$'s in \eqref{def. xi}
\begin{align*}
t^{m, k}_1 = \frac{2k}{2^{m+1}}, \quad
&t^{m, k}_2 = \frac{2k+1}{2^{m+1}}, \quad
t^{m, k}_3 = \frac{2k+2}{2^{m+1}}, \\
t^{m', k'}_1 = \frac{2k'}{2^{m'+1}}, \quad
&t^{m', k'}_2 = \frac{2k'+1}{2^{m'+1}}, \quad
t^{m', k'}_3 = \frac{2k'+2}{2^{m'+1}}.
\end{align*}

\subsection{Construction of fake fractional Brownian processes}    \label{subsec. mimicking fBM}

Similar to Section \ref{sec. mimicking BM}, we shall replace Gaussian Schauder coefficients with other random variables to construct a non-Gaussian stochastic process $Y^H$ for any given Hurst index $H \in (0, 1)$, which also satisfies the following properties of fBM $B^H$ with the same index $H$.

\begin{theorem}[Construction of fake fBM]    \label{thm.fake-fBM}
For a balanced and complete refining sequence of partition $\pi$ of $[0,1]$ and a fixed $H \in (0, 1)$, consider the stochastic process $Y^H$ defined on $[0, 1]$ via Schauder representation along $\pi$
\begin{equation*}
	Y^H(t) = \sum_{m=0}^{\infty} \sum_{k \in I_m} \theta^{Y^H}_{m,k} e^{\pi}_{m,k}(t),
\end{equation*}
where $\{\theta_{m,k}^{Y^H}\}$ are random variables satisfying
\begin{enumerate} [label=(\alph*)]
	\item $\mathbb{E}[\theta^{Y^H}_{m,k}] = 0$ and $\mathbb{E}[\theta^{Y^H}_{m,k} \theta^{Y^H}_{m',k'}] = \sigma_{m, k, m', k'}$ of \eqref{eq.fbm.theta.cov},
	\item there exist some positive constants $C_1$ and $C_2$ such that the inequality
	\begin{equation}    \label{con.theta bound delta H}
		\mathbb{P} \Big[ \vert \theta^{Y^H}_{m, k}\vert  \vert\pi^{m+1} \vert^{ \frac{1}{2}-H} > \delta  \Big]
		\le C_1 \exp \Big( -\frac{\delta^2}{2C_2} \Big)
	\end{equation}
	holds for every $\delta > 0$, $k \in I_m$, and $m \ge 0$.
\end{enumerate}
Then, the process $Y^H$ satisfies the following properties (fBM i-vi).
\begin{enumerate} [label=(fBM \roman*)]
	\item $Y^H$ has a continuous path almost surely.
	\item $\mathbb{E}[Y^H(t)] = 0$ and $2 \mathbb{E}[Y^H(t) Y^H(s)] = |t|^{2H}+|s|^{2H}-|t-s|^{2H}$ for any $t, s \in [0, 1]$.
	\item Increments of $Y^H$ over disjoint intervals are correlated (except for the case $H = 1/2)$:
\[2\mathbb{E}[(Y^H(t)-Y^H(s))(Y^H(v)-Y^H(u))] = \vert t-u \vert^{2H} + \vert s-v \vert^{2H} - \vert t-v \vert^{2H} - \vert s-u\vert^{2H}\]
	for every $0 \le s \le t \le u \le v \le 1$.
	\item Almost every path of $Y^H$ belongs to $C^{H-}([0, 1], \mathbb{R})$. 
	\item The family of Schauder coefficients $(\theta_{m, k}^{Y^H})$ has the same first two moments as the family $(\theta_{m, k}^{B^H})$ of real fBM.
	\item If the partition sequence $\pi$ consists of uniform partitions (i.e. $\forall \;n \in \mathbb{N}; \;|\pi_n| = \underline{\pi^n}$   such as dyadic, triadic partitions), we have the following convergence of the scaled quadratic variation along $\pi$ for every $t \in [0, 1]$:
	\begin{equation}    \label{eq. scaled qv convergence}
		\lim_{n \rightarrow \infty} \mathbb{E} \bigg[ \Big\lfloor \frac{\vert \pi^n \vert}{t} \Big\rfloor^{1-2H} \sum_{\substack{t^n_i, t^n_{i+1} \in \pi^n \\ t^n_{i+1} \le t}} \big|Y^H(t^n_{i+1})-Y^H(t^n_{i})\big|^2 \bigg] = t^{2H}.
	\end{equation}
\end{enumerate}
\end{theorem}

\begin{remark}  \label{rem. critical Holder fBM}
Property (fBM iv) can also be strengthened as the following under an additional assumption on the Schauder coefficients. Suppose that there exists a positive constant $C$ satisfying
\begin{equation*}
	\mathbb{P} \Big[ \sup_{m, k} \Big( |\theta^{Y^H}_{m, k}(\omega)||\pi^{m+1}|^{\frac{1}{2}- H} \Big) \le C \Big] = 1,
\end{equation*}
then, $Y^H$ of Theorem \ref{thm.fake-fBM} satisfies

\noindent (fBM iv') \textit{Almost every path of $Y^H$ has the critical H\"older exponent $H$.}
\end{remark}

\begin{remark} [Scaled quadratic variation]
In the convergence \eqref{eq. scaled qv convergence}, the expression $\left\lfloor \frac{\vert \pi^n \vert}{t} \right\rfloor$ represents the number of partitions points of $\pi^n$ on the interval $[0, t]$, which is just the reciprocal of the number of summands in the next summation, thus the left-hand side can be interpreted as an averaged scaled quadratic variation over the uniform partitions.

In the case of real fBM $B^H$, we have a stronger mode of convergence than the one in \eqref{eq. scaled qv convergence}:
\begin{equation*}
	\Big\lfloor \frac{\vert \pi^n \vert}{t} \Big\rfloor^{1-2H} \sum_{\substack{t^n_i, t^n_{i+1} \in \pi^n \\ t^n_{i+1} \le t}} \big\vert B^H(t^n_{i+1})-B^H(t^n_{i})\big\vert^2 \xlongrightarrow{n \rightarrow  \infty} t^{2H} \quad \text{in probability},
\end{equation*}
which can be proven by the self-similarity, stationary increments property of fBM, and the ergodic theorem (see \cite[Chapter ~1.18]{Mishura:book} or \cite{Rogers:1997} for the proof along a uniform partition sequence, but it can be easily generalized to other uniform refining sequences of partitions such as dyadic). This convergence is one of the three characteristics of fBM, studied in \cite{Mishura:Valkeila} (as an extension of the L\'evy's characterization to fBM). However, for the process $Y^H$ we construct in the following example, its increments are not stationary (but covariance stationary or weak-sense stationary) and the self-similarity is not guaranteed. Thus, we stated the weaker convergence of the form \eqref{eq. scaled qv convergence} in Theorem \ref{thm.fake-fBM}; nonetheless, we observe in Figures 7-9 of Supplement B \cite{suppB} that the fake fBMs in the next subsection exhibit this stronger convergence.
\end{remark}

\begin{proof}
We show each property from \textit{(fBM i)} to \textit{(fBM vi)} one by one.

\textbf{(fBM i)} Condition (b) is from condition (ii) of Lemma~\ref{lem: continuity condition}, for $\epsilon = H$, thus $Y^H$ has a continuous path almost surely.

\textbf{(fBM ii)} Proposition \ref{prop. finiteness condition} proves that the family of coefficients $(\theta^B_{m, k})$ (of real fBM) satisfies the condition \eqref{ineq. moment bounds} for $\ell = 2$. Furthermore, Cauchy-Schwarz inequality and condition (a) deduce
\begin{equation*}
	\mathbb{E} \big| \theta^{Y^H}_{m_1, k_1} \theta^{Y^H}_{m_2, k_2} \big| \le \sqrt{\mathbb{E}[(\theta^{Y^H}_{m, k})^2] \mathbb{E}[(\theta^{Y^H}_{m', k'})^2]}
	= \sigma_{m, k}  \sigma_{m', k'},
\end{equation*}
and the argument in the proof of Proposition \ref{prop. finiteness condition} yields that the condition \eqref{ineq. moment bounds} for $\ell = 2$ is satisfied for the family $(\theta^{Y^H}_{m, k})$ as well. Theorem~\ref{thm. finite moments} now proves the property, since we matched the mean and covariance of the two families of Schauder coefficients $(\theta^{B^H}_{m, k})$ and $(\theta^{Y^H}_{m, k})$ in condition (a).

\textbf{(fBM iii)} This is again straightforward from \textit{(fBM ii)}.

\textbf{(fBM iv)}
This follows from Corollary \ref{cor. Holder continuity}.

\textbf{(fBM iv') of Remark \ref{rem. critical Holder fBM}}
We have the bounds
\begin{equation*}
	\sup_{m,k} \big(|\theta^{Y^H}_{m,k}| |\pi^{m+1}|^{\frac{1}{2}-\alpha} \big)
	\le C \sup_{m} \big( |\pi^{m+1}|^{H-\frac{1}{2}} |\pi^{m+1}|^{\frac{1}{2}-\alpha} \big)
	= C \sup_{m} |\pi^{m+1}|^{H-\alpha},
\end{equation*}
where the last term is finite if and only if $\alpha \le H$, therefore Theorem \ref{main.thm} proves \textit{(fBM iv')}.

\textbf{(fBM v)} This is obvious from the construction of $Y^H$.

\textbf{(fBM vi)} From \textit{(fBM ii)}, the first two moments of $Y^H$ coincide with that of fBM $B^H$. Thus, we have
\begin{equation*}
	\mathbb{E} \big\vert Y^H(t^n_{i+1}) - Y^H(t^n_i) \big\vert^2 
	= \mathbb{E} \big\vert B^H(t^n_{i+1}) - B^H(t^n_i) \big\vert^2
	= \vert t^n_{i+1} - t^n_{i} \vert^{2H}
	= \vert \pi^n \vert^{2H},
\end{equation*}
for any consecutive partition points $t^n_{i}, t^n_{i+1} \in \pi^n$. The expectation on the left-hand side of \eqref{eq. scaled qv convergence} is now equal to
\begin{equation*}
	\Big\lfloor \frac{\vert \pi^n \vert}{t} \Big\rfloor^{1-2H} \sum_{\substack{t^n_i, t^n_{i+1} \in \pi^n \\ t^n_{i+1} \le t}} \vert \pi^n \vert^{2H}
	= \Big\lfloor \frac{\vert \pi^n \vert}{t} \Big\rfloor^{1-2H} \times \vert \pi^n \vert^{2H} \times \Big\lfloor \frac{t}{\vert \pi^n \vert} \Big\rfloor,
\end{equation*}
and the last expression converges to $t^{2H}$ as $n \to \infty$.
\end{proof}

Thanks to Theorem \ref{thm.fake-fBM}, we have the following L\'evy type construction of fractional Brownian Motions.

\begin{corollary}   \label{cor: real fBM}
For a balanced and complete refining sequence of partition $\pi$ of $[0,1]$ and a fixed $H \in (0, 1)$, the stochastic process $Z^H$ defined by the Schauder representation along $\pi$
\begin{equation*}
	Z^H(t) = \sum_{m=0}^{\infty} \sum_{k \in I_m} \theta^{Z^H}_{m,k} e^{\pi}_{m,k}(t), \qquad t \in [0, 1],
\end{equation*}
where the coefficients $\{\theta_{m,k}^{Z^H}\}$ are normally distributed with mean zero and covariance given as \eqref{eq.fbm.theta.cov}, is fractional Brownian motion with Hurst index $H$ up to indistinguishability.
\end{corollary}

\begin{proof}
    A fBM $B^H$ with Hurst index $H$ admits the Schauder representation in \eqref{eq. non-standard normal coefficients}. Since the first two moments uniquely determine the distribution of any Gaussian processes, \textit{(fBM i)} and \textit{(fBM ii)} are sufficient conditions for $Y^H$ in Theorem \ref{thm.fake-fBM} to be real fBM, provided that $Y^H$ is a Gaussian process (see  \cite[Definition ~1.1.1]{Biagini_Hu_Oksendal_Zhang}). Therefore, we provide the following converse statement.
Suppose that we have a sequence of (correlated) standard normal random variables $\{Z_{m, k}\}$ such that the family $\big\{\theta^{Z^H}_{m, k} := \sigma_{m, k}Z_{m, k}\big\}$ has covariance
$\mathbb{E}\big[\theta^{Z^H}_{m, k} \theta^{Z^H}_{m', k'}\big] = \sigma_{m, k, m', k'}$ as in condition (a) of Theorem \ref{thm.fake-fBM}. From the bound $\sigma_{m, k} \le K |\pi^{m+1}|^{H-\frac{1}{2}}$ for some constant $K > 0$, the condition \eqref{con.theta bound delta H} is also satisfied:
\begin{equation*}
\mathbb{P} \Big[ \vert \theta^{Z^H}_{m, k}\vert  \vert\pi^{m+1} \vert^{\frac{1}{2}-H} > \delta  \Big]
\le
\mathbb{P} \Big[ K |Z_{m, k}| > \delta \Big]
\le \exp \Big( -\frac{\delta^2}{2K^2} \Big), \quad \forall \, \delta > 0, ~ k \in I_m, ~ m \ge 0.
\end{equation*}
The result now follows from Theorem \ref{thm.fake-fBM}.
\end{proof}

\subsection{Examples} \label{subsec. examples fBM}

\begin{example} [Coefficients with Uniform and Beta distributions]\label{ex. all fake fBM}
As in Examples~\ref{ex. uniform coefficients} and \ref{ex. beta coefficients}, we shall use the uniform random variable $U \sim \text{Uniform}(-\sqrt{3}, \, \sqrt{3})$ and two (centered \& scaled) beta random variables $B_1, B_2$ of \eqref{def.beta.dist} in order to construct fake fBM. For any $H \in (0, \frac{1}{2})\cup (\frac{1}{2},1)$, we recall the expressions in \eqref{eq.var.fBM}, \eqref{eq.cov.fBM} and consider three Schauder representations along the dyadic partition $\mathbb{T}$:
\begin{equation}    \label{all coefficients fBM}
	Y^H_i(t) = \sum_{m=0}^{\infty} \sum_{k \in I_m} \theta^{Y^H_i}_{m, k} e^{\mathbb{T}}_{m, k}(t), \qquad i = 0, 1, 2,
\end{equation}
such that the coefficients $\theta^{Y^H_i}_{m, k}$ are distributed as
\begin{equation*}
	\theta^{Y^H_0}_{m, k} \sim \Sigma_{m, k}U, \qquad 
	\theta^{Y^H_1}_{m, k} \sim \Sigma_{m, k}B_1, \qquad
	\theta^{Y^H_2}_{m, k} \sim \Sigma_{m, k}B_2,
\end{equation*} 
and the covariance between $\theta^{Y^H_i}_{m, k}$ and $\theta^{Y^H_i}_{m', k'}$ is equal to $\Sigma_{m, k, m', k'}$, for each $i = 0, 1, 2$. Then, all three processes $Y^H_i$ satisfy the conditions of Theorem \ref{thm.fake-fBM}. Figure 7 of Supplement B \cite{suppB} provides sample paths of fractional Brownian process $B^H(t)-Zt$ in \eqref{eq. non-standard normal coefficients} and three processes in \eqref{all coefficients fBM}.
\end{example}

\begin{example} [Mixed coefficients]    \label{ex. mixed fake fBM}
To mimic fBM, one can mix normal and uniform random variables as in Example~\ref{ex. mixed coefficients}. Recalling the coefficients $\theta^{Y_{mix}}_{m, k}$ of \eqref{eq.mixed coefficients}, we consider the Schauder representation for a fixed $H \in (0, 1)$ along the dyadic partitions
\begin{equation}    \label{mixed coefficients fBM}
	Y^H_{mix}(t) = \sum_{m=0}^{\infty} \sum_{k \in I_m} \theta^{Y^H_{mix}}_{m, k} e^{\mathbb{T}}_{m, k}(t)
\end{equation}
such that the coefficients $\theta^{Y^H_{mix}}_{m, k}$ are distributed as $\Sigma_{m, k} \theta^{Y_{mix}}_{m, k}$ and covariance between $\theta^{Y^H_{mix}}_{m, k}$ and $\theta^{Y^H_{mix}}_{m', k'}$ is again equal to $\Sigma_{m, k, m', k'}$. Then, $Y^H_{mix}$ satisfies the assumption of Theorem \ref{thm.fake-fBM} and hence has the properties \textit{(fBM i - vi)}. Figure~8 of Supplement B \cite{suppB} provides a sample path of the fake fractional Brownian process $Y^H_{mix}$ with Hurst index $H = 0.25$ up to level $n = 12$, together with their $1/H$-th variation and scaled quadratic variation.
\end{example}

Figure 9 of Supplement B \cite{suppB} provides another sample path of smoother fake fBM in Examples \ref{ex. all fake fBM} and \ref{ex. mixed fake fBM} for $H = 0.75$, where $\mathbb{E}\vert N(0, 1) \vert^{0.75} \approx 0.797$.

\begin{remark} [Methods of simulation]
For simulating and plotting sample paths of Examples in Sections~\ref{sec. mimicking BM} and \ref{sec. mimicking fBM}, we used the R language. Since the Schauder coefficients of the examples in Section~\ref{sec. mimicking BM} are independent, random sample generating functions in R are used when simulating the coefficients. For the examples in Section~\ref{sec. mimicking fBM}, as the Schauder coefficients are correlated, we first constructed a very large copula with the covariance structure given by \eqref{eq.cov.fBM}, and randomly generated samples from the copula with different finite-dimensional distributions (uniform and beta). As this process requires a lot of computations, we only used the dyadic partition sequence and truncated up to level $n = 12$ for simulating fake fBMs, whereas we truncated up to level $n=15$ and considered non-dyadic partition sequence (Example \ref{ex. non-dyadic}) for fake Brownian motions. 
\end{remark}

\subsection{Normality tests on the marginal distributions of fake fBMs}    \label{subsec. normality fBM} 

Even though the fake fBM in the previous subsection are non-Gaussian, we will perform some normality tests as we do in Section \ref{subsec. normality BM}. Likewise before we simulated $5000$ sample paths of the fake process $Y^H_{mix}$ in Example \ref{ex. mixed fake fBM}, defined on support $[0, 1]$ along the dyadic partitions, truncated at a level $n = 12$. We randomly chose $10$ dyadic points $t_i \in (0, 1)$ and took $5000$ sample values $Y^H_{mix}(t_i)$ at each $t_i$ for $i = 1, \cdots, 10$. Then, we did the same three normality tests from section \ref{sec. mimicking BM} to find out whether those $5000$ sample values for each point $t_i$ statistically have Gaussian marginals.

As we can observe from Table 2 of Supplement B \cite{suppB}, all the $p$-values are bigger than $5\%$, meaning we should not reject the null hypothesis (with a significance level of $0.05$) stating that our sample values are drawn from a Gaussian distribution. Figure 10 of Supplement B \cite{suppB} also gives histograms and Q-Q plots of the $5000$ sample values of $Y^H_{mix}(t)$ at another two (randomly chosen dyadic) points, $t = 1240/2^{12}$ and $3881/2^{12}$. Despite the finite-dimensional distributions of the fake fBMs $Y^H_{mix}$ are not \textit{theoretically} Gaussian, the simulation study concludes that commonly used statistical tests fail to identify the fake process as non-Gaussian.

\section{Moments of Schauder coefficients}  \label{sec. moment}

In the previous sections, we only matched the first two moments of the Schauder coefficients of the fake processes with that of real fBMs in order the fake processes to have the property (fBM ii), namely, the first two (joint) moments coincide with that of fBMs. The following result generalizes this to higher-order moments. The following theorem shows that the finite (joint) moments of Schauder coefficients uniquely determine the finite (joint) moments of the process up to the same order.

\begin{theorem} [Matching moments] \label{thm. finite moments}
Consider the Schauder representations of two processes $X$ and $Y$, defined on $[0, 1]$, along the same finitely refining partition sequence $\pi$:
\begin{equation*}
	X(t) = \sum_{m=0}^{\infty} \sum_{k \in I_m} \theta^{X}_{m, k} e^{\pi}_{m, k}(t),
	\qquad 
	Y(t) = \sum_{m=0}^{\infty} \sum_{k \in I_m} \eta^{Y}_{m, k} e^{\pi}_{m, k}(t).
\end{equation*}
Suppose that the two families of (random) Schauder coefficients $\{\theta^X_{m, k}\}$ and $\{\eta^Y_{m, k}\}$ have the same finite moments up to order $\ell$ for some $\ell \in \mathbb{N}$, i.e.,
the identity 
\begin{equation}    \label{con.matching moment}
	\mathbb{E} \left[ \prod_{i=1}^n \theta^X_{m_i, k_i} \right] = \mathbb{E} \left[ \prod_{i=1}^n \eta^Y_{m_i, k_i} \right]
\end{equation}
holds for every $n \le \ell$ and arbitrary $n$ pairs $(m_i, k_i)_{i = 1}^n$ such that $m_i \in \mathbb{N}_0$, $k_i \in I_{m_i}$. If we have
\begin{align}
	&\sum_{m_1, k_1} \cdots \sum_{m_{\ell}, k_{\ell}} \mathbb{E} \bigg\vert \prod_{j=1}^{\ell} \theta^X_{m_j, k_j} \bigg\vert \bigg( \prod_{j=1}^{\ell} e^{\pi}_{m_j, k_j}(t_j)\bigg) < \infty, \quad \text{ and }	\label{ineq. moment bounds}
	\\
	&\sum_{m_1, k_1} \cdots \sum_{m_{\ell}, k_{\ell}} \mathbb{E} \bigg\vert \prod_{j=1}^{\ell} \eta^Y_{m_j, k_j} \bigg\vert \bigg( \prod_{j=1}^{\ell} e^{\pi}_{m_j, k_j}(t_j)\bigg) < \infty,		\nonumber
\end{align}
for any $\ell$ pairs $(m_i, k_i)_{i = 1}^{\ell}$ with $m_i \in \mathbb{N}_0$, $k_i \in I_{m_i}$ and $\ell$ points $t_1, \cdots, t_{\ell} \in [0, 1]$,
then the finite-dimensional distributions of $X$ and $Y$ have the same finite moments up to order $\ell$, that is,
\begin{equation*}
	\mathbb{E} \Big[ \prod_{j=1}^{\ell} X(t_j) \Big] = \mathbb{E} \Big[ \prod_{j=1}^{\ell} Y(t_j) \Big].
\end{equation*}
In particular, if $\{\theta^X_{m, k}\}$ and $\{\eta^Y_{m, k}\}$ have the same finite moments and the inequalities in \eqref{ineq. moment bounds} hold for every order $\ell \in \mathbb{N}$, then the finite-dimensional distributions of $X$ and $Y$ have the same finite moments for every order.
\end{theorem}

\begin{proof}
First, the $\ell$-th order moment of $X$ evaluated at the points $0 \le t_1 \le \cdots \le t_{\ell} \le 1$ is well-defined due to Fatou's lemma and \eqref{ineq. moment bounds}:
\begin{align*}
	\mathbb{E} \Big\vert \prod_{j=1}^{\ell} X(t_j) \Big\vert
	&= \mathbb{E} \bigg\vert \sum_{m_1, k_1} \cdots \sum_{m_{\ell}, k_{\ell}} \Big( \prod_{j=1}^{\ell} \theta^X_{m_j, k_j} \Big) \Big( \prod_{j=1}^{\ell} e^{\pi}_{m_j, k_j}(t_j) \Big) \bigg\vert
	\\
	&\le \sum_{m_1, k_1} \cdots \sum_{m_{\ell}, k_{\ell}}  \mathbb{E} \bigg\vert \prod_{j=1}^{\ell} \theta^X_{m_j, k_j} \bigg\vert \bigg( \prod_{j=1}^{\ell} e^{\pi}_{m_j, k_j}(t_j) \bigg) < \infty.
\end{align*}
Thus, we can derive from the above inequality and dominated convergence theorem that
\begin{align*}
	\mathbb{E} \Big[ \prod_{j=1}^{\ell} X(t_j) \Big]
	&= \mathbb{E} \bigg[ \sum_{m_1, k_1} \cdots \sum_{m_{\ell}, k_{\ell}} \Big( \prod_{j=1}^{\ell} \theta^X_{m_j, k_j} \Big) \Big( \prod_{j=1}^{\ell} e^{\pi}_{m_j, k_j}(t_j) \Big) \bigg]
	\\
	&= \sum_{m_1, k_1} \cdots \sum_{m_{\ell}, k_{\ell}} \bigg( \mathbb{E} \Big[ \prod_{j=1}^{\ell} \theta^X_{m_j, k_j} \Big] \Big( \prod_{j=1}^{\ell} e^{\pi}_{m_j, k_j}(t_j) \Big) \bigg).
\end{align*}
Since the last term $\prod_{j=1}^{\ell} e^{\pi}_{m_j, k_j}(t_j)$ is independent of the process $X$ and the $\ell$-th order moment of the coefficients $\{\theta^X_{m_j, k_j}\}_{j=1}^{\ell}$ can be replaced by that of $\{\eta^Y_{m_j, k_j}\}_{j=1}^{\ell}$, we can proceed backwards to show that the right-hand side is equal to $\mathbb{E} \big[ \prod_{j=1}^{\ell} Y(t_j) \big]$.
\end{proof}

We note here that the proofs of the following results in this section are given in Supplement A \cite{suppA}. Next, we show that the (real) fractional Brownian process $B^H(t)-Zt$ of \eqref{eq. non-standard normal coefficients} satisfies the condition \eqref{ineq. moment bounds} in Theorem~\ref{thm. finite moments}.

\begin{proposition}   \label{prop. finiteness condition}
For the Schauder representation \eqref{eq. non-standard normal coefficients} of fBM $B^H$ along a balanced and complete refining partition sequence $\pi$, we have
\begin{equation}    \label{eq. finiteness condition}
	\sum_{m_1, k_1} \cdots \sum_{m_{\ell}, k_{\ell}} \mathbb{E} \bigg\vert \prod_{j=1}^{\ell} \theta^{B^H}_{m_j, k_j} \bigg\vert \bigg( \prod_{j=1}^{\ell} e^{\pi}_{m_j, k_j}(t_j)\bigg) < \infty
\end{equation}
for all $\ell$ pairs $(m_i, k_i)_{i = 1}^{\ell}$ with $m_i \in \mathbb{N}_0$, $k_i \in I_{m_i}$ and $\ell$ points $t_1, \cdots, t_{\ell} \in [0, T]$ for every $\ell \in \mathbb{N}$.
\end{proposition}

The generic form of the Schauder coefficients of the fake fractional processes in Examples~\ref{ex. all fake fBM} and \ref{ex. mixed fake fBM} are distributed as $\Sigma_{m, k} R_{m, k}$, where each $R_{m, k}$ is either a random variable with compact support or a Gaussian random variable, thus having finite moments of all orders. Therefore, the same argument in the proof of Proposition~\ref{prop. finiteness condition} can be applied to show the same finiteness condition \eqref{eq. finiteness condition} for the fake processes. Theorem~\ref{thm. finite moments} then concludes that if we match the joint moments of the Schauder coefficients up to more high-order (than the second order which we did in the previous subsections) with that of the Schauder coefficients of fBM, then the fake process should have the same finite joint moments as fBMs up to the same order.

Jarque-Bera test measures how the third and fourth moments of given samples are close to those moments of normal distributions. If we match the moments of the Schauder coefficients up to the fourth order as real fBMs, then the first four moments of the fake process will theoretically coincide with those of real fBMs (see Remark \ref{rem. third, fourth moment}).

Proposition \ref{prop. finiteness condition} immediately yields the following convergence result.

\begin{corollary}   \label{cor. fBM covariance convergence}
Along a balanced and complete refining partition sequence $\pi$ of $[0,1]$, let us recall the covariances in \eqref{eq.fbm.theta.cov} and \eqref{eq.fbm.theta.cov Z} between the Schauder coefficients of fBM. The following identity always holds for any $s, t \in [0, T]$:
\begin{equation}    \label{eq.R^H(t, s)}
	\sum_{m=-1}^{n} \sum_{k \in I_m} \sum_{m'=-1}^{n} \sum_{k' \in I_{m'}}\sigma_{m, k, m', k'} e^{\pi}_{m, k}(t)e^{\pi}_{m', k'}(s) \xlongrightarrow{n \rightarrow \infty} \frac{1}{2} (|t|^{2H}+|s|^{2H}-|t-s|^{2H}).
\end{equation}
\end{corollary}

Note that the Schauder functions $e^\pi_{m,k}$ and the quantities $\sigma_{m, k, m', k'}$ of Corollary \ref{cor. fBM covariance convergence} are deterministic and depend only on the choice of the partition sequence $\pi$ but not on the paths of fBM $B^H$.

We conclude this section with the following corollary of Theorem~\ref{thm. finite moments}; matching the moments of the Schauder coefficients up to a higher order gives rise to an additional property of the fake fractional processes of $Y^H$ in Theorem \ref{thm.fake-fBM}.

\begin{corollary}\label{cor. additional conditions fBM}
Besides conditions (a) and (b) of Theorem \ref{thm.fake-fBM}, suppose that the following also holds
\begin{enumerate}
	\item [(c)] the reciprocal $1/H$ of Hurst index is an even integer, which we denote by $p := 1/H \in 2\mathbb{N}$, and the two families of Schauder coefficients $\{\theta^{Y^H}_{m, k}\}$ and $\{\theta^{B^H}_{m, k}\}$ satisfy the conditions \eqref{con.matching moment} and \eqref{ineq. moment bounds} up to order $\ell = p$, such that the finite-dimensional distributions have the same moments up to order $p$, i.e.,
	\begin{equation}    \label{eq. same moment up to p}
		\mathbb{E} \Big[ \prod_{j=1}^{n} Y^H(t_j) \Big] = \mathbb{E} \Big[ \prod_{j=1}^{n} B^H(t_j) \Big] \quad \text{for every } n \le p.
	\end{equation}
\end{enumerate}
Then, the fake fractional process $Y^H$ shares an additional property (fBM vii) with fBM $B^H$:
\begin{enumerate} [label=(fBM \roman*)]
	\setcounter{enumi}{7}
	\item we have the convergence of the (discrete) $p$-th variation for every $t \in [0, 1]$
	\begin{equation}    \label{convergence in mean c_p}
		\lim_{n \rightarrow \infty} \mathbb{E} \bigg[ \sum_{\substack{t^n_i, t^n_{i+1} \in \pi^n \\ t^n_{i+1} \le t}} \big|Y^H(t^n_{i+1})-Y^H(t^n_{i})\big|^p \bigg] = C_p t, \quad \text{ where } \quad C_p := \mathbb{E} \vert N(0, 1) \vert^p. 
	\end{equation}
\end{enumerate}
\end{corollary}

In Figures 7 and 8 (b) of Supplement B \cite{suppB}, the graphs of the quartic variations ($1/4$-th variations) of fake processes seem to be linear, even though we didn't match the Schauder coefficients up to order $p = 1/H = 4$. Especially, the quartic variation of $Y^{H}_{mix}$ in Figure 8(b) has almost the same slope as that of real fBM.

Furthermore, the graphs of the $3/4$-th variations in Figure 9(b) of Supplement B \cite{suppB} also look linear, though the Hurst index $H = 3/4$ is not a reciprocal of an even integer. Therefore, we \textbf{conjecture} that the $(1/H)$-th variation of fake fractional processes along a balanced, completely refining partition sequence $\pi$ will be linear with slope equal to $C_{1/H}:= \mathbb{E}[|N(0, 1)|^{\frac{1}{H}}]$, if we match the first $m= \Big\lceil \frac{1}{H} \Big \rceil$ moments of the Schauder coefficients with that of real fBM with Hurst index $H$.

\section{Conclusion}\label{sec: conclusion}

Using (generalized) Schauder representation, we constructed a stochastic process with any H\"older regularity by controlling the growth of Schauder coefficients along given partition sequence on a finite interval (Theorem \ref{main.thm}). Moreover, we can make the finite moments of the process as we want, by controlling the joint moments between the Schauder coefficients (Theorem \ref{thm. finite moments}). These results are combined to suggest a new way of constructing stochastic processes which are statistically indistinguishable from Brownian motion or fractional Brownian motions.

Our results bring two important messages for mathematical modeling with stochastic processes, especially in finance. First, when measuring the roughness of a given function or process (e.g. for stock volatility), H\"older regularity and variation index are two different notions. Therefore, when observing a sample path from a process, we should not measure its H\"older roughness by computing $p$-th variation. Secondly, due to their full-fledged theory, Brownian and fractional Brownian motions have been widely used for modeling stochastic phenomena in various fields. However, our construction of fake processes suggests that a `model-free' approach is important; we cannot conclude from a given dataset that the underlying process follows certain dynamics involving (fractional) Brownian motions, even though it exhibits the same pathwise and statistical properties as those Gaussian processes.

Our method of construction in Sections \ref{sec. mimicking BM} and \ref{sec. mimicking fBM} can also be useful for generating diverse sets of training data for deep neural networks modeling financial time series (e.g. \cite{hamdouche2023generative}), which we leave as future work.

\begin{funding}
    E. Bayraktar is supported in part by the National Science Foundation under grant DMS-2106556 and by the Susan M. Smith Professorship.
\end{funding}
\begin{supplement}
    \stitle{Supplement A \cite{suppA} : Additional proofs}
    \sdescription{This supplement material includes the additional proofs of the results in this paper.}
\end{supplement}
\begin{supplement}
    \stitle{Supplement B \cite{suppB} : Additional Figures and Tables}
    \sdescription{This supplement material contains additional Figures and Tables for Section \ref{sec. mimicking BM} and \ref{sec. mimicking fBM}.}
\end{supplement}
\bibliographystyle{imsart-number.bst}
\bibliography{pathwise1.bib}

\includepdf[pages=-]{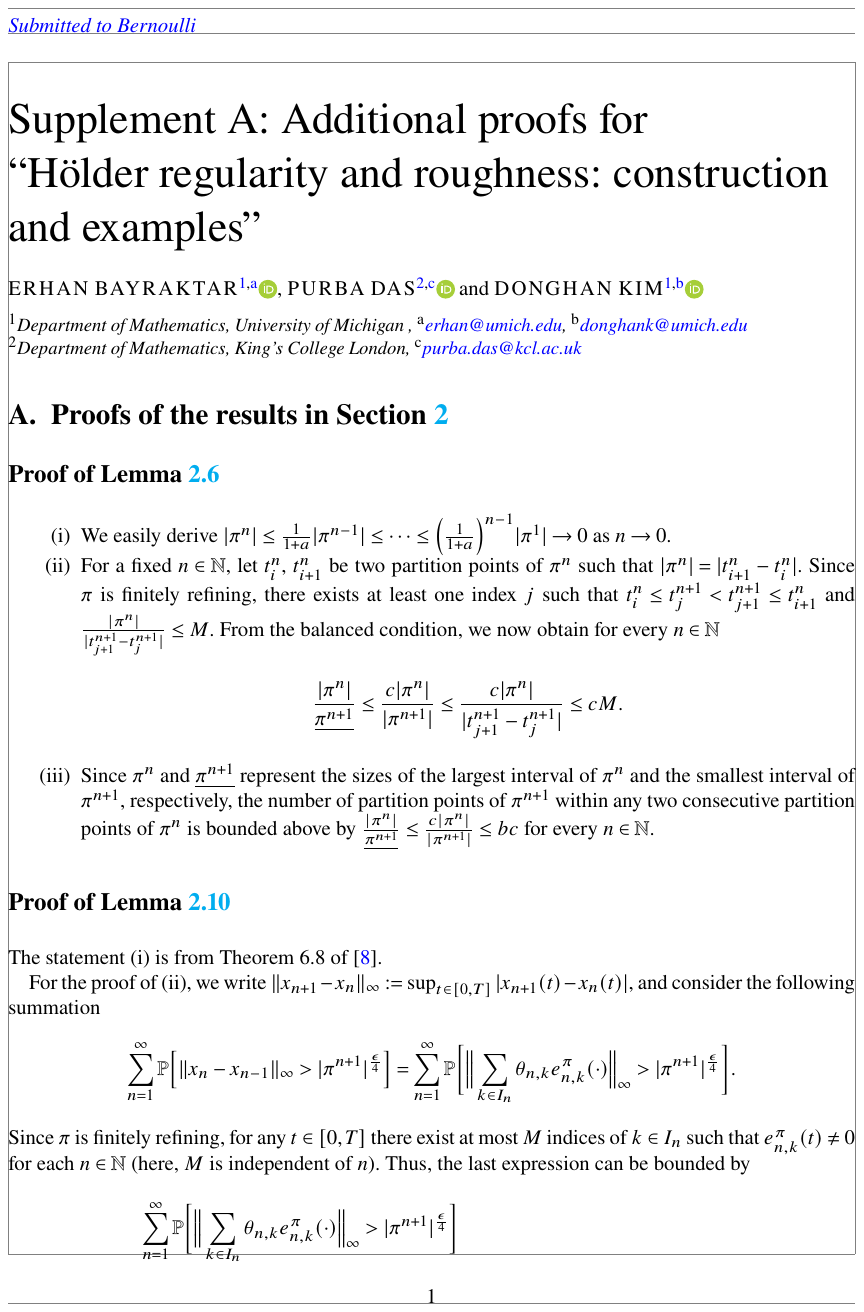}
\includepdf[pages=-]{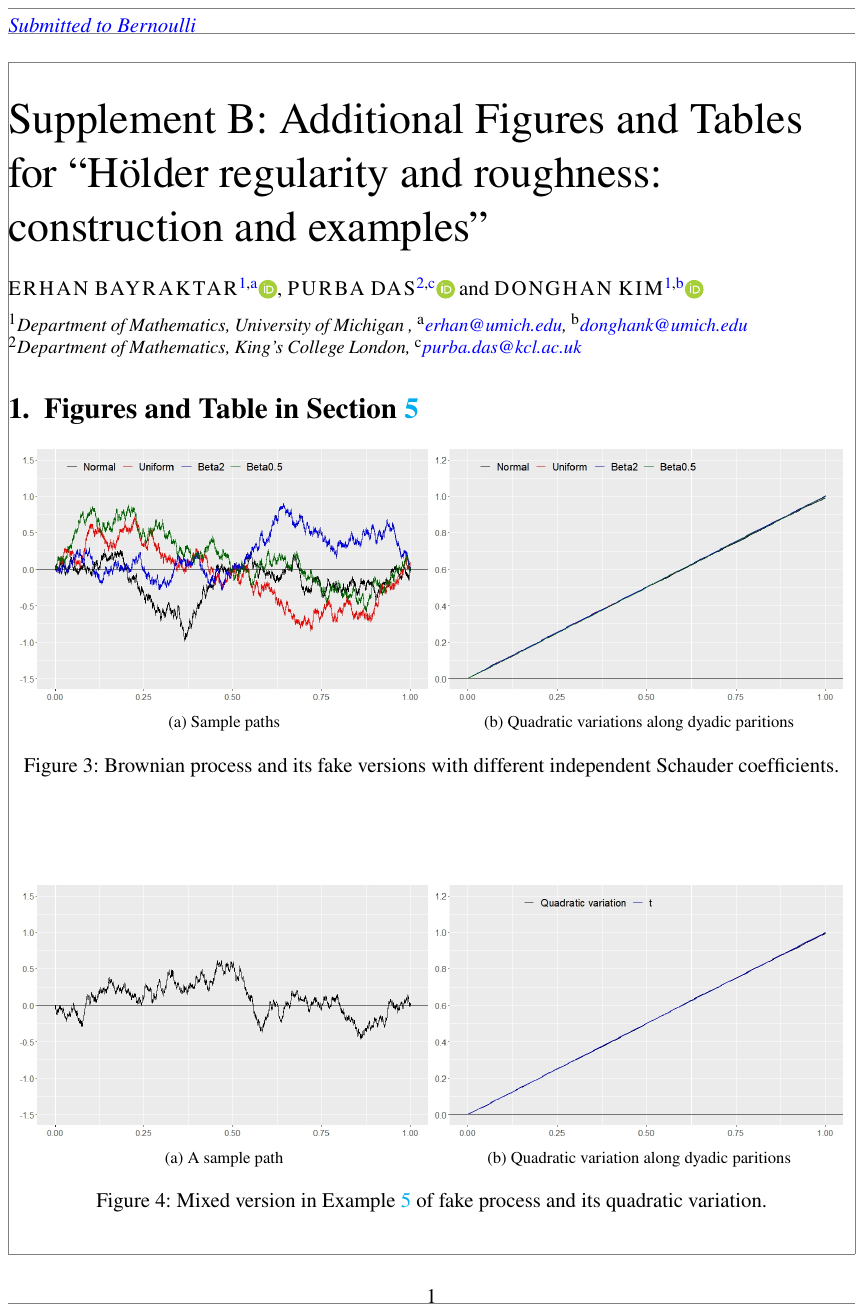}

\end{document}